\documentclass[12pt]{article}
\usepackage[usenames,dvipsnames]{color}
\usepackage{caption}

\usepackage{geometry}
\usepackage{graphicx,amsthm,amsmath,amssymb,color,srcltx,xr,xargs,a4wide,bbm,natbib}
\usepackage[colorlinks]{hyperref}
\usepackage[shortlabels]{enumitem}
\usepackage{pdfsync}
\usepackage{multirow}
\usepackage{longtable}
\usepackage{csquotes}

\usepackage{threeparttable}
\usepackage{pdflscape, geometry}

\DeclareMathOperator{\E}{E}
\DeclareMathOperator{\Var}{Var}
\DeclareMathOperator{\Cov}{Cov}

\DeclareMathOperator{\rank}{rank}

\newtheorem{theorem}{Theorem}[section]

\newtheorem{lemma}[theorem]{Lemma}

\newtheorem{corollary}[theorem]{Corollary}
\usepackage{framed}
\usepackage{listings}
\newcommand{\ignore}[1]{}
\newcommand{\spaceD}{{\mathbb{D}}}

\newcommand{\blind}{0}

\usepackage{lipsum}

\usepackage[doublespacing,nodisplayskipstretch]{setspace}

\begin{document}

\title{Testing for Change in Stochastic Volatility with Long Range Dependence}
\if0\blind
{\author{Annika Betken\thanks{Research supported by the German National Academic Foundation and Collaborative Research Center SFB 823 {\em Statistical modelling of nonlinear dynamic processes}.
}\\
Faculty of Mathematics\\
Ruhr-Universit\"at Bochum\\
{\tt annika.betken@rub.de}  \and Rafa{\l} Kulik\\
Department of Mathematics\\ and Statistics\\
University of Ottawa\\
{\tt rkulik@uottawa.ca}
}
} \fi

\bigskip
\maketitle
\newpage
\doublespacing
\begin{abstract}
In this paper, change-point problems for long memory stochastic volatility models are considered.
A general testing problem which includes various alternative hypotheses is discussed.
Under the hypothesis of stationarity
 the limiting behavior of  CUSUM- and Wilcoxon-type test statistics is derived.  In this context,  a limit theorem for the two-parameter empirical process of long memory stochastic volatility time series is proved.
In particular, it is shown that the asymptotic distribution of  CUSUM test statistics may not be affected by long memory, unlike Wilcoxon test statistics which are typically influenced by long range dependence.
To avoid the estimation of nuisance parameters in applications,  the usage of self-normalized test statistics is proposed.
The theoretical results  are accompanied by simulation studies which characterize the finite sample behavior of the considered testing procedures when testing for changes in mean, in variance, and in the tail index. 
\end{abstract}

\noindent%
{\it Keywords:}  long memory stochastic volatility; change point tests; empirical process limit theorem; self-normalization %3 to 6 keywords, that do not appear in the title
\vfill

\newpage

\section{Introduction}
One of the most often observed phenomena in financial data is the fact that   the
 log-returns of stock-market prices  appear to be uncorrelated, whereas the absolute log-returns or squared log-returns tend to be  correlated  or even exhibit long range dependence. Another characteristic of financial time series is  the existence of heavy tails in the sense that the marginal tail distribution behaves like a regularly varying function. Both of these features of empirical data can be covered by the so-called long memory stochastic volatility model (LMSV, in short), with its original version introduced in \cite{breidt:crato:delima:1998}.
For this we  assume that the data generating process $\{X_j,j\geq 1\}$ satisfies
\begin{align}
X_j=\sigma(Y_j)\varepsilon_j,\; \ \ j\geq 1\;,\label{eq: LMSV}
\intertext{
where
\begin{itemize}
\item[$\bullet$] $\{\varepsilon_j,j\geq 1\}$ is an i.i.d. sequence with mean zero;
\item $\{Y_j,j\geq 1\}$ is a stationary, long range dependent (LRD) Gaussian process;
\item[$\bullet$] $\sigma(\cdot)$ is a non-negative measurable function, not equal to $0$.
\end{itemize}
Note that within this model long memory results from the  subordinated Gaussian sequence $\{Y_j,j\geq 1\}$ only.
More precisely, we assume that
$\{Y_j,j\geq 1\}$ admits a linear representation with respect to an  i.i.d. Gaussian sequence $\{\eta_j,j\geq 0\}$ with $\E (\eta_1)=0$, $\Var(\eta_1)=1$, i.e.}
%\begin{align}\label{eq:def-y}
Y_j=\sum_{k=1}^{\infty}c_k\eta_{j-k}\;,\ \ j\geq 1\;,
\intertext{with $\sum_{k=1}^\infty c_k^2=1$
and}
\gamma_{Y}(k)=\Cov(Y_j, Y_{j+k})=k^{-D}L_{\gamma}(k),
\end{align}
where $D\in \left(0, 1\right)$ and $L_{\gamma}$ is slowly varying at infinity. Also, we assume that
\begin{itemize}
\item[$\bullet$] $\{(\varepsilon_j,\eta_j),j\geq 1\}$ is a sequence of i.i.d. vectors.
\end{itemize}
The above set of assumptions we will call collectively \textbf{LMSV model}. 

The tail behavior  of the sequence $\{X_j,j\geq 1\}$ can be related to the tail behavior  of $\{\varepsilon_j, j\geq 1\}$ or $
\{\sigma(Y_j), j\geq 1\}$ or both. Here, we will specifically assume that
\begin{itemize}
\item the random variables $\{\varepsilon_j, j\geq 1\}$ have  a marginal distribution with regularly varying right tail, i.e. $\bar{F}_{\varepsilon}(x):=P(\varepsilon_1>x)=x^{-\alpha}L(x)$ for some $\alpha>0$ and a slowly varying function $L$, such that the following tail balance condition holds:
\begin{align*}
\lim_{x\to\infty}\frac{P(\varepsilon_1>x)}{P(|\varepsilon_1|>x)}=p=1-\lim_{x\to\infty}\frac{P(\varepsilon_1<-x)}{P(|\varepsilon_1|>x)}
\end{align*}
for some $p\in (0,1]$;
\item we have
$E\left[\sigma^{\alpha+\delta}(Y_1)\right]<\infty$
for some $\delta >0$.
\end{itemize}
Under these conditions,  it follows by Breiman's Lemma (see \cite[Proposition 7.5]{resnick:2007}) that
\begin{align}\label{eq:breiman}
P(X_1>x)\sim E[\sigma^{\alpha}(Y_1)]P(\varepsilon_1>x),
\end{align}
i.e. the process $\{X_j,j\geq 1\}$ inherits the tail behavior  from $\{\varepsilon_j,j\geq 1\}$.

We would like to point out here that
in the literature the usage of the term LMSV often presupposes that the sequences $\{Y_j,j\geq 1\}$ and $\{\varepsilon_j,j\geq 1\}$ are  independent. In this paper, we will consider a more general model:  instead of claiming mutual independence of  $\{Y_j,j\geq 1\}$ and $\{\varepsilon_j,j\geq 1\}$, we only assume that $\{{\left(\eta_j, \varepsilon_j\right)},j\geq 1\}$ is a sequence of independent  random vectors.
Especially, this implies that for a fixed index
$j$ the random variables $\varepsilon_j$ and $Y_j$ are independent, while $Y_j$ may depend on  $\left\{\varepsilon_i, i<j\right\}$.
In many cases, this version of the LMSV model is referred to as LMSV with leverage.

\subsection{Change-point detection under long memory}
One of the problems related to financial data is to detect structural changes in a time series $\{X_j,j\geq 1\}$. Although the  problem has been extensively studied for independent random variables (see an excellent book \cite{csorgo:horvath:1997}) or in case of weakly dependent data, the issue has not been fully resolved for time series with long memory. The researchers focused rather on justifying that the observed long range dependence is real or that it is due to changes in weakly dependent sequences (so-called spurious long memory); see e.g. \cite{berkes:horvath:kokoszka:shao:2006} and Section 7.9.1 of \cite{beran:kulik:2013} for further references on the latter issue.

As for the testing changes in long memory sequences, one  of the first paper seems to be \cite{horvath:kokoszka:1997}, where the authors showed that long range dependence affects the asymptotic behavior of the CUSUM statistics for changes in the mean.
For the general testing problem with a change in the marginal distribution under  the alternative hypothesis, the paper \cite{giraitis:leipus:surgailis:1996} considers Kolmogorov - Smirnov type
change-point tests and change-point estimators for  long memory moving average processes.
For changes in the mean of long range dependent time series \cite{horvath:kokoszka:1997} and \cite{betken:2016b} consider estimators for the change-point location based on CUSUM and Wilcoxon statistics.
Under the assumption of  converging change-point alternatives  in LRD time series,  the asymptotic behavior of Kolmogorov-Smirnov and Cram\'{e}r-von Mises type test statistics  has also been investigated by \cite{tewes:2015}.
 Likewise, in \cite{dehling:rooch:taqqu:2013} the authors show that the Wilcoxon test is always affected by long memory. In fact, in the case of Gaussian long memory data, the asymptotic relative efficiency of the Wilcoxon test and the CUSUM test is 1.

In case of long range dependence, the normalization and  the limiting distribution of test statistics typically depend on unknown multiplicative factors or parameters related to the dependence structure of the data generating processes.
To bypass estimation of these quantities, the concept of self-normalization has recently been applied to several testing procedures in change-point analysis: 
\cite{shao:zhang:2010} define a self-normalized Kolmogorov-Smirnov test statistic that
serves to identify changes in the mean of short range dependent time series.
 \cite{shao:2011} adopted the same approach to define an alternative normalization for the
CUSUM test;  \cite{betken:2016} considers a self-normalized version of the
Wilcoxon change-point test proposed by \cite{dehling:rooch:taqqu:2013}.

The main message from the above discussion is that long memory typically affects the limiting behaviour of different statistics used in change-point testing.
We refer also to Section 7.9 of \cite{beran:kulik:2013} for further results on change-point detection for long memory processes.  

\subsection{Change-point detection for LMSV}

In this paper, we study CUSUM and Wilcoxon tests for the LMSV model and discuss particular cases of testing for changes in the mean, in the variance and in the tail index.  Although the variance and the tail index can be regarded as the mean of transformed random variables, we observe different effects for each of the three quantities. In particular,
the main findings of our paper are as follows:
\begin{itemize}
\item[{\rm A-1:}] CUSUM tests for a change in the mean in  LMSV time series are typically not affected by long memory (see Corollary \ref{Cor:partial sum} and Example \ref{sec:change-in-mean}). This is different than the findings in \cite{horvath:kokoszka:1997} for subordinated Gaussian processes;
\item[{\rm A-2:}]  Wilcoxon tests for a change in the mean of LMSV models are typically affected by long memory (see Corollary \ref{cor:wilcoxon} and Example \ref{sec:change-in-mean}). This is in line with the findings for subordinated Gaussian processes; cf. \cite{dehling:rooch:taqqu:2013}.

\item[{\rm B:}] CUSUM and Wilcoxon tests for a change in the variance of LMSV models are typically affected by long memory; see Section \ref{sec:change-in-variance}.

\item[{\rm C:}] CUSUM and Wilcoxon tests for a change in the tail index of LMSV models are typically affected by long memory; see Section \ref{sec:examples-tail}
\end{itemize}
The paper is structured as follows:
in Section \ref{sec:prel} we collect some results on subordinated Gaussian processes. In Section \ref{sec:change-point problem} we discuss the change-point problem. In particular, we consider CUSUM  and Wilcoxon test. The asymptotic behavior of the former is a direct consequence of the existing results, while the latter requires a new theorem (see Theorem \ref{thm:wilcoxon}) on the limiting behavior of empirical processes based on LMSV data. Section \ref{sec:examples} is devoted to examples in case of testing for changes in the mean, in the variance and in the tail. Since the test statistics and/or limiting distributions involve unknown quantities, self-normalization is considered in Section \ref{sec:self-norm}. In fact, in Section \ref{sec:simulations} we perform simulation studies and indicate that self-normalization provides robustness. %Finally, in Section \ref{sec:data} we analyse change-point detection for real data.

\section{Preliminaries}\label{sec:prel}

\subsection{Some properties of subordinated Gaussian sequences}\label{sec:gaussian}
The main tool for studying the asymptotic behavior of  subordinated Gaussian sequences is the Hermite expansion.
For a stationary, long range dependent Gaussian process $\{Y_j,j\geq 1\}$ and a measurable function $g$ such that $E\left(g^2(Y_1)\right)<\infty$ the corresponding Hermite expansion is defined by
\begin{align*}
g(Y_1)-\E\left(g(Y_1)\right)=\sum\limits_{q=m}^{\infty}\frac{J_q(g)}{q!}H_q(Y_1) \;,
\end{align*}
where $H_q$ is the $q$-th Hermite polynomial,
$$J_q(g)=\E \left(g(Y_1)H_q(Y_1)\right)$$
and $$m=\inf\left\{q\geq 1 \left|\right. J_q(g)\not=0\right\}\;.$$
The integer $m$ is called the Hermite rank of $g$ and we refer to $J_q(g)$ as the $q$-th Hermite coefficient  of $g$.

We will also consider the Hermite expansion of the function class
$\left\{1_{\left\{g(Y_1)\leq x\right\}}-F_{g(Y_1)}(x), \ x\in \mathbb{R}\right\}\;,$
where $F_{g(Y_1)}$ denotes the distribution function of $g(Y_1)$. For fixed $x$,  we have
\begin{align*}
&1_{\left\{g(Y_1)\leq x\right\}}-F_{g(Y_1)}(x)=\sum\limits_{q=m}^{\infty}\frac{J_q(g;x)}{q!}H_q(Y_1)\;
\end{align*}
with
$
J_q(g;x)=\E \left(1_{\left\{g(Y_1)\leq x\right\}}H_q(Y_1)\right)\;.
$
The Hermite rank corresponding to this function class is defined by $m=\inf_x m(x)$, where $m(x)$ denotes the Hermite rank of $1_{\left\{g(Y_1)\leq x\right\}}-F_{g(Y_1)}(x)$. We refer to \cite{beran:kulik:2013} for further details. 

 The asymptotic behavior of partial sums of subordinated Gaussian sequences is characterized in \cite{taqqu:1979}. Due to the functional non-central limit theorem in that paper,
\begin{align}\label{eq:fclt-partial-sums}
\frac{1}{d_{n,m}}\sum\limits_{j=1}^{\lfloor nt\rfloor}g(Y_j)\Rightarrow \frac{J_m(g)}{m!}Z_m(t), \ 0\leq t\leq 1,
\end{align}
where $Z_m(t)$, $0\leq t\leq 1$,  is an $m$-th order Hermite process,
\begin{align*}
d_{n,m}^2=\Var\left(\sum\limits_{j=1}^nH_m(Y_j)\right)\sim c_mn^{2-mD}L^m(n), \ c_m=\frac{2 m!}{(1-Dm)(2-Dm)},
\end{align*}
 and the convergence holds in $\spaceD([0,1])$, the space of functions that are right-continuous with left limits. In fact, the limiting behavior in (\ref{eq:fclt-partial-sums}) is the same as that of the corresponding partial sums based on $\{H_m(Y_j),j\geq 1\}$:
\begin{align}\label{eq:hermite-conv}
J_m(g)\frac{1}{d_{n,m}}\sum\limits_{j=1}^{\lfloor nt\rfloor}H_m(Y_j)\Rightarrow \frac{J_m(g)}{m!}Z_m(t), \ 0\leq t\leq 1\;.
\end{align}
Moreover, the functional central limit theorem for the empirical processes was established in \cite{dehling:taqqu:1989}.
Specifically,
\begin{align}\label{eq:weak-reduction-principle}
\sup\limits_{-\infty\leq x\leq \infty}\sup\limits_{0\leq t\leq 1}d_{n,m}^{-1}\left\{\lfloor nt\rfloor\left(G_{\lfloor nt\rfloor}(x)- \E\left( G_{\lfloor nt\rfloor}(x)\right)\right)-J_m(g;x)\sum\limits_{j=1}^{\lfloor nt\rfloor}H_m(Y_j)\right\}\overset{P}{\longrightarrow}0\;,
\end{align}
where
$G_l(x)=\frac{1}{l}\sum_{j=1}^l1_{\left\{g(Y_j)\leq x\right\}}$
is the empirical distribution function of the  sequence $\{g(Y_j),j\geq 1\}$ and \enquote{$\overset{P}{\longrightarrow}$} denotes convergence in probability. Thus, the empirical process
$$
d_{n,m}^{-1}\lfloor nt\rfloor\left(G_{\lfloor nt\rfloor}(x)- \E\left( G_{\lfloor nt\rfloor}(x)\right)\right)\;, \ \ x\in \left[-\infty,\infty\right], \ t\in [0,1]\;,
$$
converges in $\spaceD([-\infty,\infty]\times [0,1])$ to $ J_m(g;x) Z_m(t)$.

We refer the reader to \cite{dehling:taqqu:1989}, \cite{taqqu:1979} and \cite{beran:kulik:2013} for more details.

\section{Change-point problem}\label{sec:change-point problem}
Given the observations $X_1, \ldots, X_n$ and a function $ \psi$, we define $\xi_j=\psi(X_j)$, $j=1, \ldots, n$, and we consider the testing problem:
\begin{align*}
&H_0: \E(\xi_1)=\cdots=\E(\xi_n) \;, \ \  \\
&H_1: \exists \ k\in\left\{1, \ldots, n-1\right\}\ \  \mbox{\rm such that }\E (\xi_1)=\cdots=\E(\xi_{k})\not=\E(\xi_{k+1})=\cdots=\E(\xi_n)\;.
\end{align*}
We choose  $\psi$ according to the specific change-point problem considered. Possible choices include:
\begin{itemize}
\item $\psi(x)=x$ in order to detect changes in the mean of the observations $X_1, \ldots, X_n$  (change in location);
\item $\psi(x)=x^2$ in order to detect changes in the variance of the observations $X_1, \ldots, X_n$ (change in volatility);
\item $\psi(x)=\log(x^2)$ or $\psi(x)=\log(|x|)$ in order to
detect changes in the index $\alpha$ of heavy-tailed observations (change in the tail index).
\end{itemize}
It is obvious that $\psi(x)=x$ and $\psi(x)=x^2$ lead to testing for a change in the mean and a change in the variance, respectively. The choice $\psi(x)=\log(|x|)$ requires an additional comment. We note that (\ref{eq:breiman}) describes only the asymptotic tail behavior  of $X_1$. For the purpose of this paper, we shall pretend that $P\left(|X_1|>x\right)=c^{\alpha}x^{-\alpha}$, $x>c$, for some $c>0$. Then the maximum likelihood estimator of $(1/\alpha)$, the reciprocal of the tail index, is
\begin{equation}
\frac{1}{n}\sum_{j=1}^n\log\left(|X_j|/c\right)\;.
\end{equation}
This estimator is used in the CUSUM test statistic. To resolve the problem of change-points in the tail index in full generality, we need to employ a completely different technique, based on the so-called tail empirical processes (see \cite{kulik:soulier:2011}). This will be done in a subsequent paper.

In any case,  the following test statistics may be applied  in order to decide on the change-point problem $(H_0, H_1)$:
\begin{itemize}
\item The CUSUM test rejects the hypothesis for large values of the test statistic $C_n=\sup\limits_{0\leq\lambda\leq 1}C_n(\lambda)$, where
\begin{align}\label{eq:CUSUM-stat}
C_n(\lambda)=\left|\sum\limits_{j=1}^{\lfloor n\lambda\rfloor}\psi(X_j)-\frac{\lfloor n\lambda\rfloor}{n}\sum\limits_{j=1}^n\psi(X_j)\right|\;.
\end{align}
\item The Wilcoxon test rejects the hypothesis for large values of the test statistic $W_n=\sup\limits_{0\leq\lambda\leq 1}W_n(\lambda)$, where
\begin{align}\label{eq:Wilcoxon-stat}
W_n(\lambda)=\left|\sum\limits_{i=1}^{\lfloor n\lambda\rfloor}\sum\limits_{j=\lfloor n\lambda\rfloor +1}^n\left( 1_{\left\{\psi(X_i)\leq \psi(X_j)\right\}}-\frac{1}{2}\right)\right|\;.
\end{align}
\end{itemize}
The goal of this paper is to obtain limiting distributions for the CUSUM and the Wilcoxon test statistic
in case of time series that follow the LMSV model.
%In order to do this, we proceed as follows:
%\begin{enumerate}
%\item In Section \ref{sec:lm} we discuss regularly varying long memory time series models and their basic properties;
%\end{enumerate}

\subsection{CUSUM Test for LMSV}\label{sec:CUSUM}

In order to determine the asymptotic behavior of the CUSUM test statistic computed with respect to the observations $\psi(X_1), \ldots, \psi(X_n)$, we have to consider the partial sum process $\sum_{j=1}^{\lfloor nt\rfloor}\left(\psi(X_j)-\E\left(\psi(X_j)\right)\right)$.

For the observations $X_1, \ldots, X_n$ that satisfy the LMSV model, the asymptotic behavior  of the partial sum process is described by Theorem 4.10 in \cite{beran:kulik:2013} and hence is stated without  proof.
In order to formulate the result, we introduce the following notation:
\begin{align*}
\mathcal{F}_{j}=\sigma\left(\varepsilon_{j}, \varepsilon_{j-1}, \ldots, \eta_{j}, \eta_{j-1}, \ldots\right),
\end{align*}
i.e. $\mathcal{F}_j$ denotes the $\sigma$-field generated by the random variables $\varepsilon_{j}, \varepsilon_{j-1}, \ldots, \eta_{j}, \eta_{j-1}, \ldots$.
Due to this construction,  $\varepsilon_j$ is independent of $\mathcal{F}_{j-1}$ and $Y_j$ is $\mathcal{F}_{j-1}$-measurable.

\begin{theorem}\label{thm:partial-sums}
Assume that  $\{X_j,j\geq 1\}$ follows the LMSV model.  Furthermore, assume  that  $\E (\psi^2(X_1))<\infty$. Define the function $\Psi$ by $\Psi(y)=\E\left(\psi(\sigma(y)\varepsilon_1)\right)$. Denote by $m$ the Hermite rank of $\Psi$ and by $J_m(\Psi)$ the corresponding Hermite coefficient.
\begin{enumerate}
\item If $\E(\psi(X_1)\left|\right. \mathcal{F}_0)\neq 0$ and $mD<1$, then
\begin{align*}
\frac{1}{d_{n,m}}\sum_{j=1}^{\lfloor nt\rfloor}\left(\psi(X_j)-\E(\psi(X_j))\right)\Rightarrow  \frac{J_m(\Psi)}{m!}Z_m(t)\;, \ \ t\in [0,1]\;,
\end{align*}
in $\spaceD([0,1])$.
\item  If $\E(\psi(X_1)\left|\right. \mathcal{F}_0)=0$, then
\begin{align*}
\frac{1}{\sqrt{n}}\sum_{j=1}^{\lfloor nt\rfloor}\psi(X_j)\Rightarrow   \sigma B(t)\;, \ \ t\in [0,1]\;,
\end{align*}
in $\spaceD([0,1])$,
where $B$ denotes a Brownian motion process and $\sigma^2= \E (\psi^2(X_1))$.
\end{enumerate}
\end{theorem}
As an immediate consequence of Theorem \ref{thm:partial-sums}, we obtain the asymptotic distribution of the CUSUM statistic.
\begin{corollary}\label{Cor:partial sum}
Assume that the assumptions of Theorem \ref{thm:partial-sums} hold.
\begin{enumerate}
\item If $\E(\psi(X_1)\left|\right. \mathcal{F}_0)\neq 0$ and $mD<1$, then
\begin{align}\label{eq:CUSUM:lrd}
\frac{1}{d_{n,m}}\sup\limits_{0\leq \lambda\leq 1}C_n(\lambda)\Rightarrow  \frac{\left|J_m(\Psi)\right|}{m!}\sup\limits_{0\leq t\leq 1}\left|Z_m(t)-tZ_m(1)\right|\;.
\end{align}
\item  If $\E(\psi(X_1)\left|\right. \mathcal{F}_0)=0$
\begin{align*}
\frac{1}{\sqrt{n}}\sup\limits_{0\leq \lambda\leq 1}C_n(\lambda)\Rightarrow  \sigma\sup\limits_{0\leq t\leq 1}\left|B(t)-tB(1)\right|\;,
\end{align*}
where $B$ denotes a Brownian motion process and $\sigma^2= \E (\psi^2(X_1))$.
\end{enumerate}
\end{corollary}
It is important to note that the Hermite rank of $\Psi$ does not necessarily correspond to the Hermite rank of $\sigma$; see Section \ref{sec:examples}.

\subsection{Wilcoxon test for LMSV}

For subordinated Gaussian time series  $\{g(Y_j),j\geq 1\}$, where $\{Y_j,j\geq 1\}$ is a stationary Gaussian LRD process and $g$ is a measurable function, the asymptotic distribution of the Wilcoxon test statistic $W_n$
is derived from the limiting behavior  of the two-parameter empirical process
\begin{align*}
\sum\limits_{j=1}^{\lfloor nt\rfloor}\left(1_{\left\{g(Y_j)\leq x\right\}}-F_{g(Y_1)}(x)\right)\;, \ \ x\in \left[-\infty,\infty\right]\;, \ t\in [0,1]\;,
\end{align*}
where $F_{g(Y_1)}$ denotes the distribution function of $g(Y_1)$; see \cite{dehling:rooch:taqqu:2013}.

In order to determine the asymptotic distribution  of the Wilcoxon test statistic for the LMSV model,   we need to establish an analogous result for the stochastic volatility process $\{X_j,j\geq 1\}$, i.e.
our preliminary goal is to prove a limit theorem for the two-parameter empirical process
\begin{align*}
G_n(x, t)=\sum_{j=1}^{\lfloor nt\rfloor}\left(1_{\left\{\psi(X_j)\leq x\right\}}-F_{\psi(X_1)}(x)\right),
\end{align*}
 where now $F_{\psi(X_1)}$ denotes the distribution function of $\psi(X_1)$ with $X_1=\sigma(Y_1)\varepsilon_1$.
To state the weak convergence, we introduce the following notation:
$$\Psi_x(y)=P\left(\psi(y\varepsilon_1)\leq x\right)\;.$$

 \begin{theorem}\label{thm:wilcoxon}
 Assume that  $\{X_j,j\geq 1\}$ follows the LMSV model.
 Moreover, assume that
 \begin{align}\label{eq:integrability}
 \int\frac{d}{du}P\left(\psi(u\varepsilon_1)\leq x\right)du<\infty\;.
 \end{align}
 Let $m$  denote the Hermite rank of the class  $\left\{1_{\left\{\sigma(Y_1)\leq x\right\}}-F_{\sigma(Y_1)}(x), \ x\in \mathbb{R}\right\}$. If $mD<1$, then
\begin{align}
\frac{1}{d_{n,m}}G_n(x, t)\Rightarrow \frac{J_m(\Psi_x \circ \sigma)}{m!}Z_m(t)\;,  \  x \in \left[-\infty,\infty\right], \ t\in \left[0, 1\right], \label{eq:limit_of_G}
\end{align}
in $\spaceD\left([-\infty, \infty]\times [0, 1]\right)$ .
 \end{theorem}
The proof of this theorem is given in Section \ref{Sec:proof}. At this moment, we derive the asymptotic distribution of the Wilcoxon statistics.
\begin{corollary}\label{cor:wilcoxon}
Under the conditions of Theorem \ref{thm:wilcoxon}
\begin{align*}
\frac{1}{nd_{n,m}}\sup\limits_{\lambda \in [0, 1]}W_n(\lambda)\Rightarrow \left|\int J_m(\Psi_x\circ \sigma)dF_{\psi(X_1)}(x)\right|\frac{1}{m!}\sup\limits_{\lambda \in [0, 1]}\left|Z_m(\lambda)-\lambda Z_m(1)\right|\;.
\end{align*}
\end{corollary}
\begin{proof}[Proof of Corollary \ref{cor:wilcoxon}]
According to \cite{dehling:rooch:taqqu:2013}, the asymptotic distribution of the Wilcoxon test statistic can be derived directly from the limit of the two-parameter empirical process if the sequence $\{X_j,j\geq 1\}$ is ergodic. Ergodicity is obvious, since $X_j$ can be represented as a measurable function of the i.i.d. vectors $\{(\eta_j,\varepsilon_j),j\geq 1\}$.
\end{proof}
\subsection{Proof of Theorem \ref{thm:wilcoxon}}\label{Sec:proof}
To prove Theorem \ref{thm:wilcoxon},
we consider the following decomposition:
\begin{align*}
&G_n(x, t)\\
&=\sum\limits_{j=1}^{\lfloor nt\rfloor}\left(1_{\left\{\psi(X_j)\leq x\right\}}-\E\left(1_{\left\{\psi(X_j)\leq x\right\}}\left|\right. \mathcal{F}_{j-1}\right)\right)
+\sum\limits_{j=1}^{\lfloor nt\rfloor}\left(\E\left(1_{\left\{\psi(X_j)\leq x\right\}}\left|\right. \mathcal{F}_{j-1}\right)-F_{\psi(X_1)}(x)\right)\\
&=:M_n(x, t)+R_n(x, t).
\end{align*}
It will be shown that $n^{-1/2}M_n(x, t)=\mathcal{O}_P(1)$ uniformly in $x,t$, while $d_{n, m}^{-1}R_n(x, t)$ converges in distribution to the limit process in formula
\eqref{eq:limit_of_G}. Theorem \ref{thm:wilcoxon} then follows because $\sqrt{n}=\text{o}(d_{n, m})$.

\paragraph{Martingale part.}
For fixed $x$, the following lemma characterizes the asymptotic behavior of the  martingale part $M_n(x, t)$.
We write
\begin{align*}
M_n(t):=M_n(x, t)=\sum\limits_{j=1}^{\lfloor nt\rfloor}\zeta_j(x)
\end{align*}
with $\zeta_j(x)=1_{\left\{\psi(X_j)\leq x\right\}}-\E\left(1_{\left\{\psi(X_j)\leq x\right\}}\left|\right. \mathcal{F}_{j-1}\right)$.
\begin{lemma}\label{lem:uniform_conv._in_t}
Under the conditions of Theorem \ref{thm:wilcoxon}, we have for every $x$,
\begin{align*}
\frac{1}{\sqrt{n}}M_n(t)\Rightarrow  \beta(x) B(t)\;, \ \  t\in [0,1]\;,
\end{align*}
in $\spaceD([0,1])$,
where $B$ denotes a Brownian motion process and $\beta^2(x)= \E(\zeta_1^2(x))$.
\end{lemma}

\begin{proof}
Define
$$\zeta_{nj}=n^{-\frac{1}{2}}\zeta_j(x)=X_{nj}(x) - \E(X_{nj}(x)\left|\right. \mathcal{F}_{j-1})$$
with $X_{nj}(x)=n^{-\frac{1}{2}}1_{\left\{\psi(X_j)\leq x\right\}}$.
In order to show convergence in $\spaceD([0,1])$, we apply the functional martingale central limit theorem as stated in Theorem 18.2 of \cite{billingsley:1999}.
Therefore, we have to show that
 \begin{align*}
 \sum\limits_{j=1}^{\lfloor nt\rfloor}\E\left(\zeta_{nj}^2\left|\right. \mathcal{F}_{j-1}\right)\Rightarrow \beta(x) t
 \end{align*}
 for every $t$ and that
 \begin{align*}
\lim_{n\to\infty}\sum\limits_{j=1}^{\lfloor nt\rfloor}\E\left(\zeta_{nj}^21_{\left\{|\zeta_{nj}|\geq\epsilon\right\}}\right)=0
\end{align*}
for every $t$ and $\epsilon>0$ (Lindeberg condition).
In order to show that the Lindeberg condition holds, it suffices to show that
\begin{align}\label{eq:dvoretzky}
\lim_{n\to\infty}\sum\limits_{j=1}^{\lfloor nt\rfloor}\E\left(X^2_{nj}(x)1_{\left\{|X_{nj}(x)|\geq\frac{\epsilon}{2}\right\}}\right)= 0
\end{align}
due to Lemma 3.3 in \cite{dvoretzky:1972}.
As the indicator  function is bounded, the above summands vanish for sufficiently large $n$ and hence (\ref{eq:dvoretzky}) follows.

Furthermore, the random variable $\E\left(\zeta_j^2(x)\left|\right.\mathcal{F}_{j-1}\right)$  can be considered as a measurable function of the random variable $Y_j$ and therefore as a function of $\eta_{j-1}, \eta_{j-2}, \ldots$.
As a result, $\E\left(\zeta_j^2(x)\left|\right.\mathcal{F}_{j-1}\right)$ is an ergodic sequence and it follows by the ergodic theorem that
\begin{align*}
\frac{1}{n}\sum\limits_{j=1}^{\lfloor nt\rfloor}\E\left(\zeta_j^2(x)\left|\right.\mathcal{F}_{j-1}\right)
=\frac{\lfloor nt\rfloor}{n}\frac{1}{\lfloor nt\rfloor}\sum\limits_{j=1}^{\lfloor nt\rfloor}\E\left(\zeta_j^2(x)\left|\right.\mathcal{F}_{j-1}\right)
\overset{P}{\longrightarrow}t\E(\zeta_1^2(x))
\end{align*}
for every $t$.
\end{proof}
The next lemma establishes uniform boundedness of the two-parameter process.
\begin{lemma}\label{lem:2-param_tightness}
Under the conditions of Theorem \ref{thm:wilcoxon}, we have
\begin{align*}
\frac{1}{\sqrt{n}}M_n(x, t)= \mathcal{O}_P(1)
\end{align*}
in $\spaceD(\left[-\infty,\infty\right]\times [0,1])$.
\end{lemma}

The (technical) proof of this lemma can be found in Section S1 of the supplementary material.
\paragraph{Long memory part.}
Finally, we prove weak convergence of the long memory part $R_n(x,t)$.
\begin{lemma}
Under the conditions of Theorem \ref{thm:wilcoxon},
\begin{align*}
\frac{1}{d_{n,m}}R_n(x, t)\Rightarrow \frac{J_m(\Psi_x \circ \sigma)}{m!}Z_m(t)\;
\end{align*}
in $\spaceD\left([-\infty, \infty]\times [0, 1]\right)$ .
\end{lemma}

\begin{proof}
Note that
\begin{align*}
\E\left(1_{\left\{\psi(X_j)\leq x\right\}}\left|\right. \mathcal{F}_{j-1}\right)=\E\left(1_{\left\{\psi(\sigma(Y_j)\varepsilon_j)\leq x\right\}}\left|\right. \mathcal{F}_{j-1}\right)=\Psi_x(\sigma(Y_j) )
\end{align*}
because $Y_j$ is $\mathcal{F}_{j-1}$-measurable and $\varepsilon_j$ is independent of $\mathcal{F}_{j-1}$. Furthermore,
$\E\left(\Psi_x(\sigma(Y_j) )\right)=F_{\psi(X_1)}(x)$, where $F_{\psi(X_1)}$ denotes the distribution function of $\psi(X_1)=\psi(\sigma(Y_1)\varepsilon_1)$.
Hence,
\begin{align*}
R_n(x, t)=\lfloor nt\rfloor \int \Psi_x(u)d\left(G_{\lfloor nt\rfloor}-\E G_{\lfloor nt\rfloor}\right)(u),
\end{align*}
where
$G_l(u)=\frac{1}{l}\sum_{j=1}^l1_{\left\{\sigma(Y_j)\leq u\right\}}$
is the empirical distribution function of the  sequence $\{\sigma(Y_j),j\geq 1\}$.
We have
\begin{align*}
&d_{n,m}^{-1}R_n(x, t)\\
&= -\Bigg\{\int\frac{d}{du}P\left(\psi(u\varepsilon_1)\leq x\right)d_{n,m}^{-1}\left\{\lfloor nt\rfloor\left(G_{\lfloor nt\rfloor}(u)- \E G_{\lfloor nt\rfloor}(u)\right)-\frac{J_m(\sigma;u)}{m!}\sum\limits_{j=1}^{\lfloor nt\rfloor}H_m(Y_j)\right\}du\Bigg\}\\
&\quad - \Bigg\{\int \frac{d}{du}P\left(\psi(u\varepsilon_1)\leq x\right)d_{n,m}^{-1}\frac{J_m(\sigma;u)}{m!}\sum\limits_{j=1}^{\lfloor nt\rfloor}H_m(Y_j)du\Bigg\}=: I_1(x,t)+I_2(x,t),
\end{align*}
where $m$  denotes the Hermite rank of the class   $\left\{1_{\left\{\sigma(Y_1)\leq x\right\}}-F_{\sigma(Y_1)}(x), \ x\in \mathbb{R}\right\}$ and
\begin{align*}
J_m(\sigma; y)=\E\left(1_{\left\{\sigma(Y_1)\leq y\right\}}H_m(Y_1)\right)\;.
\end{align*}
Using the reduction principle (\ref{eq:weak-reduction-principle}) with $g=\sigma$ and the integrability condition (\ref{eq:integrability}), we conclude that
the first summand converges to $0$ in probability, uniformly in $x,t$.
Furthermore,
\begin{align*}
I_2(x,t)= -d_{n,m}^{-1}\sum\limits_{j=1}^{\lfloor nt\rfloor}H_m(Y_j)\Bigg\{\int \frac{d}{du}P\left(\psi(u\varepsilon_1)\leq x\right)\frac{J_m(\sigma; u)}{m!}du\Bigg\}\;.
\end{align*}
%We have
%\begin{align*}
%d_{n,m}^{-1}\sum\limits_{j=1}^{\lfloor nt\rfloor}H_m(Y_j)\Rightarrow Z_m(t), \ t\in \left[0,1\right].
%\end{align*}
Denoting by $\varphi$ the standard normal density,
integration by parts yields
\begin{align*}
&\int \frac{d}{du}P\left(\psi(u\varepsilon_1)\leq x\right)J_m(\sigma;u)du=-J_m(\Psi_x\circ \sigma)\;.
\end{align*}
The proof is concluded by (\ref{eq:hermite-conv}).
\end{proof}

\section{Examples}\label{sec:examples}
\subsection{Change in the mean}\label{sec:change-in-mean}
To test for a change in the mean, we choose $\psi(x)=x$. 

\noindent \textbf{CUSUM:} Recall that the function $\Psi$ in Theorem \ref{thm:partial-sums} is defined as
$\Psi(y)=\E\left(\psi(\sigma(y)\varepsilon_1)\right)$.
In this case,
$\E(\psi(X_1)\left|\right. \mathcal{F}_0)=0$.
Therefore, the CUSUM statistic converges to a Brownian bridge. Hence,
long memory in LMSV does not influence the asymptotic behavior  of the CUSUM statistic when testing for a change in the mean.

\noindent \textbf{Wilcoxon:}
Recall that $\Psi_x(y)=P\left(\psi(y\varepsilon_1)\leq x\right)$.
Using integration by parts and noting that $(d/dz)\varphi(z)=-z\varphi(z)$, we conclude
\begin{align*}
&J_1(\Psi_x\circ\sigma)=\int x\frac{d}{dz}\left(\frac{1}{\sigma(z)}\right)f_{\varepsilon}\left(\frac{x}{\sigma(z)}\right)\varphi(z)dz,
\end{align*}
where $f_{\varepsilon}$ is the density of $\varepsilon_1$ (if it exists). Here, different scenarios are possible: if $\sigma(y)=y^2$, then
the integrand is antisymmetric in $z$ for  any choice of $f_{\varepsilon}$. Hence,
$J_1(\Psi_x\circ\sigma)=0$ and some calculations show that the Hermite rank of $\Psi_x\circ\sigma$ is $2$.  If $\sigma(y)=\exp(y)$ and
$\varepsilon_1$ is Pareto distributed, i.e. for some $\alpha, c>0$
$f_{\varepsilon}(x)=
\frac{\alpha c^{\alpha}}{x^{\alpha+1}}1_{\left\{x\geq c\right\}}$,
then, as a result,
\begin{align*}
J_1(\psi_x\circ \sigma)
=\alpha c^{\alpha}x^{-\alpha}\int\exp(z\alpha)1_{\left\{ \log\left(\frac{x}{c}\right)\geq z\right\}}\varphi(z)dz.
\end{align*}
Hence, $J_1(\Psi_x\circ\sigma)\not=0$. In any case,
long memory in LMSV influences the asymptotic behavior  of the Wilcoxon statistic.

\subsection{Change in the variance}\label{sec:change-in-variance}
To test for a change in the variance, we choose $\psi(x)=x^2$.  

\noindent \textbf{CUSUM:}
Recall again that the function $\Psi$ in Theorem \ref{thm:partial-sums} is defined as
$\Psi(y)=\E\left(\psi(\sigma(y)\varepsilon_1)\right)$.
Then
$\E(\psi(X_1)\left|\right. \mathcal{F}_0)\not=0$
and hence long memory affects the limiting behavior  of the CUSUM statistic.
Moreover,
 \begin{align*}
J_m(\Psi)
&=\E (\varepsilon_1^2)\int\sigma^2(z) H_m(z)\varphi(z)dz=\E \left(\varepsilon_1^2\right) J_m(\sigma^2),
\end{align*}
i.e. the Hermite rank of $\Psi$ equals the Hermite rank of $\sigma^2$. If $mD<1$, then the limiting behavior  of the CUSUM statistic is described by (\ref{eq:CUSUM:lrd}). Hence,
long memory in LMSV influences the asymptotic behavior  of the CUSUM statistic when  testing for a change in the variance.

\noindent \textbf{Wilcoxon:}
Recall again that $\Psi_x(y)=P\left(\psi(y\varepsilon_1)\leq x\right)$.
For $x\geq 0$, we have
\begin{align*}
J_1(\Psi_x\circ\sigma)
=\int\left( \sqrt{x}\frac{d}{dz}\frac{1}{\sigma(z)}\right)\left\{f_{\varepsilon}\left(\frac{\sqrt{x}}{\sigma(z)}\right)+f_{\varepsilon}\left(-\frac{\sqrt{x}}{\sigma(z)}\right)\right\}\varphi(z)dz.
\end{align*}
If $\sigma(y)=\exp(y)$, then we can conclude by the same argument as in the case of testing for  a change in the mean that $J_1(\Psi_x\circ \sigma)\not=0$.
Hence,
long memory in LMSV time series influences the asymptotic behavior  of the Wilcoxon statistic when testing for a change in the variance.

\subsection{Change in the tail index}\label{sec:examples-tail}
To test for a change in the tail index, we choose $\psi(x)=\log(x^2)$.   

\noindent \textbf{CUSUM:}
In this case,
$\E(\psi(X_1)\left|\right. \mathcal{F}_0)\not=0$
and hence long memory affects the limiting distribution of the CUSUM statistic. Moreover,
\begin{align*}
J_m(\Psi)=2\int \log(\sigma(z))H_m(z)\varphi(z)dz=2J_m(\log\circ \sigma),
\end{align*}
so that the Hermite rank of $\Psi$ equals the Hermite rank  of $h=\log\circ \sigma$.

We note further that if $\psi(x)=\log(x^2)$, we have
 \begin{align*}
  &\frac{1}{d_{n,m}}\sum\limits_{j=1}^{\lfloor nt\rfloor}\left(\log \left(X_j^2\right)-\E\log \left(X_j^2\right)\right)\\
&=\frac{2}{d_{n,m}}\sum\limits_{j=1}^{\lfloor nt\rfloor}\left(\log \sigma(Y_j)-\E \log\sigma
(Y_j)\right)+
\frac{\sqrt{n}}{d_{n,m}}\frac{1}{\sqrt{n}} \sum\limits_{j=1}^{\lfloor nt\rfloor}\left(\log \left(\varepsilon^2_j\right)-\E\log\left(\varepsilon^2_j\right) \right).
\end{align*}
 The first summand  converges
to $2(J_m(\log\circ \sigma)/m!)Z_m(t)$ as a consequence of the functional non-central limit theorem discussed in Section \ref{sec:gaussian}. The second term is $o_P(1)$ uniformly in $t$ by  Donsker's theorem, if $\Var (\log \varepsilon_1^2)<\infty$. As a result, this observation is consistent with  Corollary \ref{Cor:partial sum}.

\noindent \textbf{Wilcoxon:}
Due to monotonicity of the logarithm, choosing $\psi(x)=\log(x^2)$ yields the same   test statistic 
as if testing for a change in variance, i.e. if we choose $\psi(x)=x^2$.
Moreover, the Wilcoxon test statistic is robust, i.e. it is designed to reduce the effect of heavy tails, so that it does not seem advisable to apply Wilcoxon-like tests when testing for a change in the tail index.

\section{Self-normalization}\label{sec:self-norm}

An application of the CUSUM test presupposes knowledge of the normalizing sequence $d_{n,m}$ (if $\E(\psi(X_1)\left|\right. \mathcal{F}_0)\neq 0$) and of the coefficients $J_m(\Psi)$ or $\sigma$ that appear in the limit of  the test statistic. Usually, these quantities are unknown.
In order to avoid estimation of the normalization and the unknown coefficients in the limit, we consider the self-normalized CUSUM test statistic with respect to the observations $\xi_j=\psi(X_j)$, $j=1, \ldots, n$. For $0<\tau_1<\tau_2<1$ it is defined by
\begin{align*}
T_n(\tau_1, \tau_2)=\sup_{k\in \left\{\lfloor n\tau_1\rfloor, \ldots,  \lfloor n\tau_2\rfloor\right\}}\left|G_n(k)\right|,
\end{align*}
where
\begin{align*}
G_n(k)
=\frac{\sum_{j=1}^k\xi_j-\frac{k}{n}\sum_{j=1}^n\xi_j}{
\left\{\frac{1}{n}\sum_{t=1}^k S_t^2(1,k)+\frac{1}{n}\sum_{t=k+1}^n S_t^2(k+1,n)\right\}^{\frac{1}{2}}
}
\end{align*}
with
$S_{t}(j, k)=\sum_{h=j}^t\left(\xi_h-\bar{\xi}_{j, k}\right), \quad
\bar{\xi}_{j, k}=\frac{1}{k-j+1}\sum_{t=j}^k\xi_t$.
The self-normalized CUSUM test  rejects the hypothesis  for
large values of the test statistic $T_n(\tau_1, \tau_2)$.
Note that the proportion of the data that is included in the calculation of the supremum
 is restricted by
  $\tau_1$ and $\tau_2$. A common  choice is $\tau_1= 1-\tau_2=0.15$; see  \cite{andrews:1993}.

In order to detect changes in the mean of (possibly) long range dependent time series, a similar test statistic has been proposed in \cite{shao:2011}.
For long memory stochastic volatility sequences the limit  of the test statistic
can be derived in the same way as in \cite{shao:2011}.
Under the assumptions of Theorem \ref{thm:partial-sums} (and if  $\E(\psi(X_1)\left|\right. \mathcal{F}_0)\neq 0$ and $mD<1$),   an application of the continuous mapping theorem to the partial sum process $\frac{1}{d_{n,m}}\sum_{j=1}^{\lfloor nt\rfloor}\left(\psi(X_j)-\E\psi(X_j)\right)$ yields
$T_n(\tau_1, \tau_2)~\overset{\mathcal{D}}{\longrightarrow}~T(m,\tau_1, \tau_2)$, where
\begin{align}\label{eq:SNCUSUM:lrd}
&T(m, \tau_1, \tau_2)=\sup\limits_{t\in \left[\tau_1, \tau_2\right]}\frac{ \left|Z_m(t)-t Z_m(1)\right|}{
\Bigl\{\int_0^{t}V_m^2(r; 0, t)dr+\int_{t}^1 V_m^2(r; t, 1)dr\Bigr\}^{\frac{1}{2}}}
\intertext{with}
&V_m(r; r_1, r_2)=Z_m(r)-Z_m(r_1)-\frac{r-r_1}{r_2-r_1}\left\{Z_m(r_2)-Z_m(r_1)\right\} \notag
\end{align}
for $r\in \left[r_1, r_2\right]$, $0< r_1< r_2< 1$.

If  $\E(\psi(X_1)\left|\right. \mathcal{F}_0)= 0$,
it follows by an application of the continuous mapping theorem to the partial sum process $\frac{1}{\sqrt{n}}\sum_{j=1}^{\lfloor nt\rfloor}\psi(X_j)$  that
$T_n(\tau_1, \tau_2)~\overset{\mathcal{D}}{\longrightarrow}~T(1, \tau_1, \tau_2)$, where $T(1,\tau_1, \tau_2)$ corresponds to \eqref{eq:SNCUSUM:lrd} with $Z_1$ denoting a fractional Brownian motion process with Hurst parameter $1/2$, i.e. under these assumptions $Z_1$ denotes a Brownian motion.
Note that in this case the limit does not depend on any unknown parameters.  The  factor $\sigma$ that appeared in the limit of the partial sum process is canceled out by self-normalization.

As an alternative to the self-normalized CUSUM test, \cite{betken:2016} proposes to use a self-normalized Wilcoxon test for the detection of changes in the mean of long range dependent time series.
For the definition of the  corresponding test statistic, we consider the ranks defined by $R_i:=\rank(X_i)=\sum_{j=1}^n1_{\{X_j\leq X_i\}}$ for $i=1,\ldots,n$.  The self-normalized two-sample Wilcoxon test statistic is then defined by
\begin{equation*}
G_n(k)
:=\frac{\sum_{i=1}^kR_i-\frac{k}{n}\sum_{i=1}^nR_i}{\bigg\{\frac{1}{n}\sum_{t=1}^k S_t^2(1,k)+\frac{1}{n}\sum_{t=k+1}^n S_t^2(k+1,n)\bigg\}^{1/2}},
\end{equation*}
where
$S_{t}(j, k):=\sum_{h=j}^t\left(R_h-\bar{R}_{j, k}\right)\ \ \text{with} \ \ \bar{R}_{j, k}:=\frac{1}{k-j+1}\sum_{t=j}^kR_t$.
The self-normalized Wilcoxon change-point test rejects the hypothesis for large values of $\max_{k\in \left\{\lfloor n\tau_1\rfloor, \ldots,  \lfloor n\tau_2\rfloor\right\}}\left|G_n(k)\right|$, where $0< \tau_1 <\tau_2 <1$.

Under the hypothesis of stationarity,
 the asymptotic distribution of the test statistic is derived in \cite{betken:2016} for long range dependent subordinated Gaussian processes.
Basically, the asymptotics of the test statistic follow from an application of the continuous mapping theorem to the process
\begin{align*}
\frac{1}{nd_{n,m}}\sum\limits_{i=1}^{\lfloor nt \rfloor}\sum\limits_{j=\lfloor nt \rfloor +1}^n\left( 1_{\left\{\psi(X_i)\leq \psi(X_j)\right\}}-\frac{1}{2}\right)\;, \ \  t\in [0,1]\;.
\end{align*}
Theorem \ref{thm:wilcoxon} implies that the above process converges in distribution to the following limit:
\begin{align*}
\int_{\mathbb{R}}J_m(\Psi_x\circ \sigma)dF_{\psi(X_1)}(x)\frac{1}{m!}\left(Z_m(t)-\lambda Z_m(1)\right)\;, \ \  t\in [0,1]\;.
\end{align*}
Under the LMSV model, the limit  of the test statistic
can therefore be derived by the same argument that proves Theorem 1 in \cite{betken:2016}. As a result,  $T_n(\tau_1, \tau_2)$ converges in distribution to $T(m,\tau_1, \tau_2)$,  i.e. the self-normalized Wilcoxon test statistic and the self-normalized CUSUM test statistic converge to the same limit.

\section{Simulations}

For all simulations we make the following specifications:
\begin{align}
X_j=\sigma(Y_j)\varepsilon_j,\; \ \ j\geq 1\;,\label{eq: LMSVsimulation_normal}
\end{align}
where
\begin{itemize}
\item  $\{\varepsilon_j,j\geq 1\}$ is an i.i.d. sequence;
\item $\{Y_j,j\geq 1\}$ is a fractional Gaussian noise sequence generated by the function \verb$fgnSim$ (\verb$fArma$ package in \verb$R$) with the Hurst parameter $H$ (note that $H$ and the memory parameter $D$ are linked by $H=1-D/2$);
\item  $\sigma(y)=\exp(y)$.
\end{itemize}

In the following, we give an overview over the main  simulation results 
for the three testing situations described in Section \ref{sec:examples}.
For all scenarios, we computed the rejection rates of the different  tests under the hypothesis and different change point alternatives.

\subsection{Change in the mean}

We investigate the finite sample performance of the CUSUM and Wilcoxon change-point tests for detecting changes in the mean of LMSV time series $\{X_j,j\geq 1\}$, i.e. we choose $\psi(x)=x$ for the test statistics  described in Section \ref{sec:change-point problem} and Section \ref{sec:self-norm}.

We consider two different scenarios:
\begin{enumerate}
\item $\{\varepsilon_j,j\geq 1\}$ are standard normal  generated  by  the function \verb$rnorm$  in \verb$R$;
\item $\{\varepsilon_j,j\geq 1\}$ follow a Pareto distribution, centered to mean zero, i.e. $\varepsilon_j=\tilde{\varepsilon}_j-\E\tilde{\varepsilon}_j$, where $\{\tilde{\varepsilon}_j,j\geq 1\}$ are Pareto distributed with parameter $\alpha$ generated  by  the function \verb$rgpd$  (\verb$fExtremes$ package in \verb$R$);
\end{enumerate}

In order to compare the  finite sample behavior of the change-point tests, we computed the empirical size and the empirical power of the testing procedures.
To determine the finite sample performance under the alternative, we simulated time series with a change-point of  height $h$ after a proportion $\tau$ of the data, i.e.
we  consider  random variables $X_j$, $j=1, \ldots, n$ with expected value $\mu_j$, $j=1\ldots, n$ such that $\mu_j=0$ for $j=1, \ldots, \lfloor n\tau\rfloor$ while  $\mu_j=h$ for $j=\lfloor n\tau\rfloor+1, \ldots, n$.
 The calculations
are based on $5,000$ realizations of time series with sample sizes $500$, $1,000$ and $2,000$.
The simulation results for $\tau=0.5$ are reported in Tables \ref{table:cp_in_mean_normal_eps} and \ref{table:cp_in_mean_(SN)}; additional simulation results for $\tau=0.25$ can be found in Section S2 of  the supplementary material.
The frequency of a type 1
error, i.e. the rejection rate under the hypothesis,
corresponds to the values in the
columns that are superscribed by \enquote{h = 0}.

 Recall that the function $\Psi$ in Theorem \ref{thm:partial-sums} is defined as
$\Psi(y)=\E\left(\psi(\sigma(y)\varepsilon_1)\right)$.
Since $\psi(x)=x$, it follows that
$\E(\psi(X_1)\left|\right. \mathcal{F}_0)=0$.
Due to Theorem \ref{thm:partial-sums},
\begin{align*}
\frac{1}{\sqrt{n}}\sum\limits_{j=1}^{\lfloor nt\rfloor}\psi(X_j)\Rightarrow  \sigma B(t)\;,
\end{align*}
where $\sigma^2=\E(\psi^2(X_1))=\E(\varepsilon_1^2\exp(2Y_1))=\Var(\varepsilon_1)\E(\exp(2Y_1))$.
 Hence,
we expect long memory in the data not to influence the asymptotic behavior of the CUSUM statistic.

 If $\{\varepsilon_j, j\geq 1\}$ are standard normal,
it can be shown that
\begin{align*}
 \left|\int J_1(\Psi_x\circ \sigma)dF_{X_1}(x)\right|
 =0 \ .
\end{align*}
Therefore,   the Wilcoxon test statistic  converges to $0$ in probability.

If $\{\varepsilon_j,j\geq 1\}$ are centered two-sided Pareto distributed with parameter $\alpha$, it follows by some basic calculations that
\begin{align*}
&f_X(x)=
\frac{1}{x}\int_{\mu}^{\infty}h_{\alpha}(x, u)du1_{\left\{x>0\right\}}
-\frac{1}{x}\int_{1}^{\mu}h_{\alpha}(x, u)du1_{\left\{x<0\right\}},
\\
&J_1(\Psi_x\circ \sigma)=
-\int_{\mu}^{\infty}h_{\alpha}(x, u)du 1_{\left\{x>0\right\}}
+\int_{1}^{\mu}h_{\alpha}(x, u)du 1_{\left\{x<0\right\}},
\end{align*}
where $\mu$ is the expected value of $\tilde{\varepsilon}_1$, i.e. $\mu:=\frac{\alpha}{\alpha-1}$, and
\begin{align*}
h_{\alpha}(x, u)=\alpha u^{-\alpha-1}\varphi\left(\frac{1}{2}\left[\log\left((u-\mu)^2\right)-\log(x^2)\right]\right)1_{\left\{u>1\right\}}.
\end{align*}
As a result, we have
\begin{align*}
\int J_1(\Psi_x\circ \sigma)dF_{\psi(X_1)}(x)
=\int_0^{\infty}\frac{1}{x}\left\{\left(\int_{1}^{\mu}h_{\alpha}(x, u)du\right)^2-\left(\int_{\mu}^{\infty}h_{\alpha}(x, u)du\right)^2\right\}dx.
\end{align*}

Since
our theoretical results do not provide a non-degenerate asymptotic distribution for the Wilcoxon test statistics, we only consider CUSUM and self-normalized CUSUM test
for the first scenario (when the $\varepsilon_j$, $j\geq 1$, are standard normally distributed). 
%\newgeometry{left=15mm,right=-15mm,top=15mm,bottom=35mm}
%\begin{landscape}
\begin{table}[htbp]
\footnotesize 
 \begin{threeparttable}
\begin{tabular}{crcccccccccccc}
 & & & \multicolumn{5}{c}{\textbf{CUSUM}}  &  &  \multicolumn{5}{c}{\textbf{self-norm. CUSUM}} \\
 \cline{4-8} \cline{10-14}\\
 {$H$} & $n$   &  & $h=0$ &  &{$h = 0.5$} &{$h = 1$} &{$h = 2$} &  & $h=0$ &  &{$h = 0.5$} &{$h = 1$} &{$h = 2$} \\
 0.6 & 500 & & 0.046  &  & 0.384 & 0.958 & 1.000 &  & 0.043 &  & 0.372 & 0.853 & 0.998 \\
   & 1000 &  & 0.045 &  & 0.728 & 1.000 & 1.000 &  & 0.047 &  & 0.627 & 0.975 & 1.000 \\
   & 2000 &  & 0.043 &  & 0.958 & 1.000 & 1.000 &  &  0.044 &  & 0.856 & 0.998 & 1.000 \\
   0.7 & 500 &  & 0.064 &  & 0.392 & 0.962 & 1.000 &  & 0.044 &  & 0.395 & 0.858 & 0.995 \\
   & 1000 &  & 0.056 &  & 0.736 & 0.999 & 1.000 &  &  0.044 &  & 0.629 & 0.974 & 1.000 \\
   & 2000 &  &  0.050 &  & 0.964 & 1.000 & 1.000 &  & 0.046 &  & 0.869 & 0.997 & 1.000 \\
  0.8 & 500 &  &  0.074&  & 0.369 & 0.969 & 1.000 &  & 0.039 &  & 0.462 & 0.860 & 0.993 \\
   & 1000 &  & 0.074 &  & 0.734 & 1.000 & 1.000 &  &  0.044 &  & 0.670 & 0.963 & 0.999 \\
   & 2000 &  &  0.072 &  & 0.966 & 1.000 & 1.000 &  &  0.048 &  & 0.851 & 0.996 & 1.000 \\
  0.9 & 500 &  & 0.095 &  & 0.311 & 0.977 & 1.000 &  & 0.045 &  & 0.583 & 0.875 & 0.980 \\
   & 1000 &  &  0.101 &  & 0.739 & 0.998 & 1.000 &  &  0.045 &  & 0.736 & 0.946 & 0.996 \\
   & 2000 &  & 0.101 &  & 0.979 & 1.000 & 1.000 &  & 0.043 &  & 0.860 & 0.987 & 1.000
 \end{tabular}
 \captionsetup{font=footnotesize }
\caption{Rejection rates of the CUSUM tests for LMSV  time series  (standard normal $\varepsilon_j,$ $j\geq 1$) of length $n$
 with Hurst parameter $H$  and a shift in the mean of height $h$ after a proportion $\tau=0.5$.
The calculations are based on 5,000 simulation runs.}
\label{table:cp_in_mean_normal_eps}
\end{threeparttable}
\end{table}
%\end{landscape}
%\restoregeometry

Note that in this case the size of the non-self-normalized CUSUM change point test tends to increase when the value of the Hurst parameter increases. Aside from that, the simulation results confirm the theoretical inference that long memory does not influence the asymptotic behavior of the CUSUM statistics when testing for a change in the mean in this situation.

A comparison of the finite sample behavior of both CUSUM tests shows that:
\begin{itemize}
\item The self-normalized test has better size properties than the non-self-normalized test.
\item The empirical power of the non-self-normalized CUSUM test exceeds the empirical power of the self-normalized CUSUM test for most parameter combinations that have been considered.
\end{itemize} 
Therefore, it is not possible to definitely conclude 	that one of the tests outreaches the other.

For the second scenario (when the $\varepsilon_j$, $j\geq 1$, are centered Pareto distributed), we considered all four testing procedures: CUSUM and self-normalized CUSUM; Wilcoxon and self-normalized Wilcoxon.
In this case, it is notable that for all  four testing procedures,  heavier tails, i.e. a decrease of the tail parameter, go along with an decrease of rejections. A change in the Hurst parameter seems to have a remarkable effect on the finite sample performance of the Wilcoxon test only: An increase of dependence goes along with a decrease of rejections.

\newgeometry{left=15mm,right=-15mm,top=15mm,bottom=35mm}
\begin{landscape}
\begin{table}
\footnotesize 
 \begin{threeparttable}
\begin{tabular}{ccrcccccccccccccc}
& & & \multicolumn{7}{c}{\textbf{CUSUM}}  &  \multicolumn{7}{c}{\textbf{Wilcoxon}} \\
 \cline{5-10} \cline{12-17}\\
& & &
& &$\alpha=2.5$ & &
& $\alpha=4$& & & &
$\alpha=2.5$& & &
$\alpha=4$\\
& $H$ & $n$ & & $h=0$ & $h=0.5$ & $h=1$ & $h=0$ & $h=0.5$ & $h=1$ & & $h=0$ & $h=0.5$ & $h=1$ & $h=0$ & $h=0.5$ & $h=1$ \\
\cline{2-17}\\
 \multirow{11}{*}{ \rotatebox{90}{non-self-normalized}}
 &0.6 & 500 & & 0.035 & 0.048 & 0.213 &  0.047 & 0.866 & 1.000 &  & 0.628& 1.000 & 1.000 & 0.802 & 1.000 & 1.000 \\
& & 1000 & & 0.034 & 0.085 & 0.745 &  0.051 & 0.993 & 1.000 & & 0.578  & 1.000 & 1.000 & 0.751 & 1.000 & 1.000 \\
& & 2000 & & 0.034 & 0.288 & 0.986 & 0.048 & 1.000 & 1.000 &  & 0.524 & 1.000 & 1.000 & 0.713 & 1.000 & 1.000 \\
&0.7 & 500 & & 0.035 & 0.050 & 0.209 &  0.053 & 0.864 & 1.000 & & 0.331 & 1.000 & 1.000 &  0.475 & 1.000 & 1.000 \\
& & 1000 &  & 0.037 & 0.087 & 0.752 &  0.058& 0.994 & 1.000 & & 0.270 & 1.000 & 1.000 & 0.384 & 1.000 & 1.000 \\
& & 2000 & & 0.039  & 0.285 & 0.983 &  0.051 & 1.000 & 1.000 &  & 0.207 & 1.000 & 1.000 & 0.300 & 1.000 & 1.000 \\
&0.8 & 500 & & 0.045 & 0.055 & 0.207 & 0.073 & 0.879 & 1.000 & & 0.191 & 0.974 & 1.000 & 0.273& 1.000 & 1.000 \\
& & 1000 &  & 0.041 & 0.108 & 0.757 & 0.066 & 0.994 & 1.000 &  & 0.144 & 0.994 & 1.000 & 0.187 & 1.000 & 1.000 \\
& & 2000 &  & 0.042 & 0.280 & 0.983 &  0.066 & 1.000 & 1.000 &  & 0.108 & 1.000 & 1.000 & 0.132 & 1.000 & 1.000 \\
&0.9 & 500 &  & 0.057  & 0.069 & 0.191 & 0.080 & 0.901 & 1.000 &  & 0.188 & 0.863 & 0.984 &0.232  & 0.995 & 1.000 \\
& & 1000 &  &  0.059 & 0.111 & 0.783 & 0.090 & 0.992 & 1.000 &  & 0.139 & 0.888 & 0.990 & 0.165 & 0.996 & 1.000 \\
& & 2000 &  & 0.064 & 0.238 & 0.984 & 0.092 & 1.000 & 1.000 &  & 0.108 & 0.917 & 0.996 &  0.121 &0.998 & 1.000\\
\\ 
 \multirow{11}{*}{ \rotatebox{90}{self-normalized}} 
 & 0.6 & 500 &  & 0.046 & 0.384 & 0.801 &  0.042 & 0.904 & 0.990 &  & 0.032 & 0.994 & 1.000 & 0.030 & 1.000 & 1.000 \\
 &  & 1000 &  & 0.049 & 0.564 & 0.909 & 0.044 & 0.973 & 0.998 &  &  0.032 & 1.000 & 1.000 & 0.030 & 1.000 & 1.000 \\
 &  & 2000 &  & 0.053 & 0.744 & 0.962 & 0.026 & 0.993 & 1.000 &  & 0.034 & 1.000 & 1.000 & 0.028 & 1.000 & 1.000 \\
 & 0.7 & 500 &  &  0.051& 0.401 & 0.801 &  0.042 & 0.902 & 0.989 &  &  0.029 & 0.950 & 1.000 &  0.021 & 1.000 & 1.000 \\
 &  & 1000 &  & 0.050 & 0.565 & 0.904 &  0.046 & 0.972 & 0.998 &  &  0.032 & 0.993 & 1.000 &  0.027 & 1.000 & 1.000 \\
 &  & 2000 &  &  0.049 & 0.744 & 0.966 & 0.042 & 0.994 & 0.999 &  &  0.037 & 1.000 & 1.000 & 0.030 & 1.000 & 1.000 \\
 & 0.8 & 500 &  & 0.045 & 0.424 & 0.804 &  0.044 & 0.899 & 0.990 &  & 0.031 & 0.776 & 0.977 & 0.021 & 0.987 & 0.999 \\
 &  & 1000 &  & 0.042 & 0.589 & 0.905 &  0.040 & 0.966 & 0.997 &  & 0.033 & 0.896 & 0.995 & 0.024 & 0.998 & 1.000 \\
 &  & 2000 &  &  0.050& 0.761 & 0.959 & 0.052 & 0.994 & 0.999 &  &  0.034 & 0.963 & 0.999 & 0.023 & 1.000 & 1.000 \\
 & 0.9 & 500 &  &  0.044 & 0.527 & 0.815 &  0.041 & 0.893 & 0.982 &  & 0.031 & 0.622 & 0.890 & 0.020 & 0.912 & 0.990 \\
 &  & 1000 &  &  0.051 & 0.650 & 0.898 & 0.042 & 0.959 & 0.997 &  & 0.039 & 0.708 & 0.936 & 0.034 & 0.956 & 0.996 \\
 &  & 2000 &  &  0.041 & 0.781 & 0.954 & 0.048 & 0.989 & 0.999 &  & 0.049 & 0.779 & 0.960 &  0.039 & 0.976 & 0.998 \\
\end{tabular}
 \captionsetup{font=footnotesize }
\caption{Rejection rates of  CUSUM and Wilcoxon  tests for  LMSV time series  (Pareto distributed $\varepsilon_j,$ $j\geq 1$) of length $n$ with Hurst parameter $H$, tail index $\alpha$  and a shift in the mean of height $h$ after a proportion $\tau=0.5$. The calculations are based on 5,000 simulation runs.}
\label{table:cp_in_mean_(SN)}
\end{threeparttable}
\end{table}
\end{landscape}
\restoregeometry

A comparison of the finite sample performance of the testing procedures  shows that:
\begin{itemize}
\item In most cases the size of the CUSUM test does not deviate much from the level of significance whereas the Wilcoxon change-point test suffers from size distortions: the empirical size of the Wilcoxon test is considerably higher than the level of significance.
\item The self-normalized change-point tests have considerably better size properties than the non-self-normalized tests. The self-normalized Wilcoxon test is rather conservative while the rejection rates of the self-normalized CUSUM test are close to the level of significance.
\item
The power of the self-normalized Wilcoxon test is higher than the power of the self-normalized CUSUM test for almost every combination of parameters.
\end{itemize}

All in all, our simulation results  give rise to choosing the self-normalized Wilcoxon test over the other testing procedures when testing for a change in the mean.

\subsection{Change in the variance}

We will now investigate the finite sample performance of the CUSUM and Wilcoxon change-point tests for detecting changes in the variance of LMSV time series $\{X_j,j\geq 1\}$, i.e. we choose $\psi(x)=x^2$ for the test statistics  described in Section \ref{sec:change-point problem} and Section \ref{sec:self-norm}. 

For this purpose, we simulate observations $X_1, \ldots, X_n$ which satisfy \eqref{eq: LMSVsimulation_normal} with
$\{\varepsilon_j,j\geq 1\}$ that follow a two-sided Pareto distribution, centered to mean zero, i.e. $\varepsilon_j=\tilde{\varepsilon}_j-\E\tilde{\varepsilon}_0$, where $\{\tilde{\varepsilon}_j,j\geq 1\}$ are Pareto distributed with parameter $\alpha>4$.
Since $\Var(\varepsilon_1)=\Var(\tilde{\varepsilon}_1)=\frac{\alpha}{\alpha -2}\frac{1}{(\alpha-1)^2}$, we have
\begin{align*}
\Var(X_1)=\E(\exp(2Y_1))\Var(\varepsilon_1)=\exp(2)\frac{\alpha}{\alpha -2}\frac{1}{(\alpha-1)^2} .
\end{align*}
Since
$\E(\psi(X_1)\left|\right. \mathcal{F}_0)\neq 0$, the limits of the CUSUM statistics depend on the Hurst parameter $H$.

For the determination of the critical values of the non-self-normalized CUSUM test,  note that
\begin{align*}
J_1(\Psi)
=\E\left(\exp(2Y_1)Y_1\right)\Var\left(\varepsilon_1\right)=2\exp(2)\frac{\alpha}{\alpha -2}\left(\frac{1}{\alpha-1}\right)^2.
\end{align*}

In order to apply the Wilcoxon test, we have to compute the value of the multiplicative factor $\left|\int J_1(\Psi_x\circ \sigma)dF_{\psi(X_1)}(x)\right|$.
For notational convenience define
$\mu:=\frac{\alpha}{\alpha -1}$.
For $x>0$, it follows by some calculations that
\begin{align*}
J_1(\Psi_x\circ\sigma)
=-\int _{1}^{\infty}
\alpha
u^{-\alpha-1}
\varphi\left(\frac{1}{2}\log(\frac{1}{x}(u-\mu)^2)\right)
du.
\end{align*}
Furthermore, it follows that the density function of $\psi(X_1)$ is given by
\begin{align*}
f_{\psi(X)}(x)=\frac{1}{2x}1_{\left\{x>0\right\}}\int _{1}^{\infty}
\alpha
u^{-\alpha-1}
\varphi\left(\frac{1}{2}\log(\frac{1}{x}(u-\mu)^2)\right)
du.
\end{align*}
We conclude that
\begin{align*}
\left|\int_{\mathbb{R}}J_1(\Psi_x\circ \sigma)dF_{\psi(X)}(x)\right|
=\int_0^{\infty}\frac{1}{2x}\left[\int _{1}^{\infty}
\alpha
u^{-\alpha-1}
\varphi\left(\frac{1}{2}\log(\frac{1}{x}(u-\mu)^2)\right)
du\right]dx.
\end{align*}
Therefore, the critical values for the Wilcoxon change-point test are based on an approximation of the above expression by numerical integration.
In order to determine the finite sample performance of the testing procedures under the hypothesis, we simulate observations  $X_1, \ldots, X_n$ with variance $\omega^2$.
For the computation of the empirical power, we added a change-point of  height $h^2$ after $\tau$ percent of the data, i.e.
we multiply $X_{\lfloor n\tau\rfloor+1}, \ldots, X_n$ by $h$ such that  $\Var(X_j)=\omega^2$ for $j=1, \ldots, \lfloor n\tau\rfloor$ while $\Var(X_j)=h^2\omega^2$ for $j=\lfloor n\tau\rfloor+1, \ldots, n$.
We computed the rejection rates of the testing procedures on the basis of $5,000$ realizations of time series with sample sizes $500$, $1,000$ and $2,000$.
The simulation results for $\tau=0.5$ are reported in Table \ref{table:cp_in_var (SN)}; additional simulation results for $\tau=0.25$ can be found in Section S2 of the supplementary material.
The frequency of a type 1
error
corresponds to the values in the
columns that are superscribed by \enquote{h = 1}.

\newgeometry{left=15mm,right=-15mm,top=15mm,bottom=35mm}
\begin{landscape}
\begin{table}
\footnotesize 
 \begin{threeparttable}
\begin{tabular}{ccrcccccccccccccc}
& & & \multicolumn{7}{c}{\textbf{CUSUM}}  &  \multicolumn{7}{c}{\textbf{Wilcoxon}} \\
 \cline{5-10} \cline{12-17}\\
& & &
& &$\alpha=4.5$ & &
& $\alpha=6$& & & &
$\alpha=4.5$& & &
$\alpha=6$\\
& $H$ & $n$ & & $h=1$ & $h=0.5$ & $h=2$ & $h=1$ & $h=0.5$ & $h=2$ & & $h=1$ & $h=0.5$ & $h=2$ & $h=1$ & $h=0.5$ & $h=2$ \\
\cline{2-17}\\
 \multirow{11}{*}{ \rotatebox{90}{non-self-normalized}} 
 & 0.6 & 500 &  & 0.464 & 0.377 & 0.954 &  0.467 & 0.417 & 0.967 &  &  0.119 & 0.988 & 0.990 &  0.130 & 0.990 & 0.990 \\
 &  & 1000 &  & 0.589 & 0.565 & 0.992 &  0.596 & 0.594 & 0.996 &  &  0.112 & 1.000 & 1.000 &  0.118 & 1.000 & 1.000 \\
 &  & 2000 &  &  0.708 & 0.774 & 0.999 &  0.694 & 0.814 & 0.999 &  &  0.108 & 1.000 & 1.000 &  0.104 & 1.000 & 1.000 \\
 & 0.7 & 500 &  &  0.330 & 0.252 & 0.821 &  0.330 & 0.252 & 0.839 &  & 0.078 & 0.808 & 0.808 & 0.078 & 0.806 & 0.809 \\
 &  & 1000 &  & 0.374 & 0.313 & 0.932 & 0.404 & 0.333 & 0.946 &  &  0.066 & 0.934 & 0.937 &  0.068 & 0.936 & 0.935 \\
 &  & 2000 &  & 0.431 & 0.416 & 0.984 &  0.443 & 0.436 & 0.987 &  & 0.060 & 0.992 & 0.992 & 0.061  & 0.990 & 0.991 \\
 & 0.8 & 500 &  &  0.235 & 0.170 & 0.623 & 0.244  & 0.166 & 0.651 &  & 0.074 & 0.529 & 0.525 &  0.074 & 0.515 & 0.528 \\
 &  & 1000 &  &  0.258 & 0.196 & 0.733 & 0.256 & 0.199 & 0.757 &  &  0.067 & 0.605 & 0.615 &  0.067 & 0.624 & 0.631 \\
 &  & 2000 &  & 0.275 & 0.226 & 0.838 &  0.271 & 0.230 & 0.848 &  &  0.060 & 0.741 & 0.738 &  0.059 & 0.746 & 0.734 \\
 & 0.9 & 500 &  & 0.179 & 0.127 & 0.440 & 0.170 & 0.132 & 0.445 &  &  0.088 & 0.410 & 0.419 &  0.088 & 0.404 & 0.406 \\
 &  & 1000 &  & 0.184 & 0.137 & 0.487 &  0.177 & 0.140 & 0.506 &  &  0.078 & 0.442 & 0.440 & 0.079 & 0.452 & 0.424 \\
 &  & 2000 &  &  0.191 & 0.135 & 0.542 &  0.191 & 0.148 & 0.540 &  &  0.072& 0.466 & 0.467 & 0.065 & 0.468 & 0.473\\
 \\
 \multirow{11}{*}{ \rotatebox{90}{self-normalized}} & 0.6 & 500 &  & 0.035 & 0.145 & 0.152 & 0.038 & 0.158 & 0.164 &  & 0.040 & 0.816 & 0.831 &  0.041 & 0.823 & 0.816 \\
 &  & 1000 &  & 0.033 & 0.206 & 0.193 &  0.035 & 0.227 & 0.221 &  & 0.043 & 0.965 & 0.963 & 0.047 & 0.962 & 0.960 \\
 &  & 2000 &  &  0.033 & 0.277 & 0.270 &  0.034 & 0.324 & 0.318 &  &  0.044 & 0.998 & 0.998 & 0.042 & 0.996 & 0.996 \\
 & 0.7 & 500 &  & 0.022 & 0.090 & 0.090 & 0.024 & 0.100 & 0.093 &  & 0.046 & 0.524 & 0.530 & 0.048 & 0.534 & 0.521 \\
 &  & 1000 &  &  0.022 & 0.125 & 0.121 & 0.025   & 0.130 & 0.135 &  &  0.044 & 0.701 & 0.695 &  0.046 & 0.698 & 0.699 \\
 &  & 2000 &  & 0.020 & 0.165 & 0.168 &  0.021 & 0.196 & 0.193 &  &  0.049 & 0.859 & 0.863 &  0.050 & 0.854 & 0.863 \\
 & 0.8 & 500 &  &  0.015 & 0.045 & 0.046 &  0.019 & 0.053 & 0.054 &  & 0.038& 0.270 & 0.274 &  0.042 & 0.270 & 0.278 \\
 &  & 1000 &  &  0.015 & 0.063 & 0.064 & 0.016 & 0.062 & 0.070 &  &  0.049 & 0.357 & 0.361 & 0.045 & 0.350 & 0.361 \\
 &  & 2000 &  &  0.015 & 0.073 & 0.086 &  0.017 & 0.099 & 0.083 &  &  0.047& 0.439 & 0.443 & 0.046 & 0.454 & 0.454 \\
 & 0.9 & 500 &  &  0.019 & 0.044 & 0.037 & 0.023 & 0.047 & 0.040 &  &  0.048 & 0.200 & 0.205 &  0.049 & 0.195 & 0.198 \\
 &  & 1000 &  & 0.019 & 0.053 & 0.052 & 0.021 & 0.064 & 0.058 &  & 0.053 & 0.228 & 0.224 & 0.050 & 0.246 & 0.228 \\
 &  & 2000 &  & 0.020 & 0.066 & 0.065 & 0.023 & 0.076 & 0.074 &  & 0.056 & 0.263 & 0.248 & 0.047 & 0.248 & 0.264
\end{tabular}
 \captionsetup{font=footnotesize }
\caption{Rejection rates of the  CUSUM and  Wilcoxon  tests for  LMSV  time series   of length $n$ with Hurst parameter $H$, tail index $\alpha$  and a shift in the variance of height $h^2$ after a proportion $\tau=0.5$. The calculations are based on 5,000 simulation runs.}
\label{table:cp_in_var (SN)}
\end{threeparttable}
\end{table}
\end{landscape}
\restoregeometry

In general, the finite sample behavior of the testing procedures does not seem to be strongly influenced by the tail parameter. An increase of dependence, i.e. increasing values of $H$, are usually followed by a decrease of the rejection rates though.

We make the following observations with respect to the behavior of the different testing procedures:
\begin{itemize}
\item The CUSUM test suffers from  severe size distortions. It rejects the hypothesis too frequently for every combination of the considered parameters $\alpha$ and $H$. In  particular, the   number of rejections increases as the number of observations increases. Apparently,  unrealistically large sample sizes are required for the asymptotics to apply.
The same observation has been made in \cite{pooter:vandijk:2004},
where
the CUSUM statistic is used to test for breaks in the variance of GARCH(1, 1) time series.
\item The self-normalized CUSUM test tends to be undersized. 
\item The Wilcoxon test rejects the hypothesis too often. Yet, the rejection rates seem to converge to the level of significance.
\item The empirical size of the self-normalized Wilcoxon test is close to the level of significance.
\item The power of both Wilcoxon-based testing procedures does not depend on $\alpha$.  Contrary to the behavior of the CUSUM test, the increase and  decrease of the variance are almost equally well detected by both Wilcoxon tests.
\item The power of the  Wilcoxon test exceeds the  power of the self-normalized Wilcoxon test for every parameter combination that has been considered.
\end{itemize}

All in all, the simulation results show that both
CUSUM tests are outperformed by the  Wilcoxon-based testing procedures:  Although the CUSUM test has high size distortions, its power can only compete with the power of the Wilcoxon test  if the shift height is relatively high. Even though the self-normalized CUSUM test does  not suffer from severe  size distortions, its power cannot compete with the power of any  other test. Obviously, CUSUM-based tests are highly unreliable when testing for a change in the variance.

\subsection{Change in the tail}
We will now investigate the finite sample performance of the CUSUM  change-point tests for detecting changes in the tail parameter $\alpha$ of LMSV time series $\{X_j,j\geq 1\}$, i.e. we choose $\psi(x)=\log(|x|)$ for the test statistics described in Section \ref{sec:change-point problem} and Section \ref{sec:self-norm}.
In this case, the Hermite rank of $\Psi(y)=\E\left(\psi(\sigma(y)\varepsilon_1)\right)$ equals $m=1$ and  $J_m(\Psi)=1$ such that the limit of the CUSUM test statistic does not depend on $\alpha$.

We simulate observations $X_1, \ldots, X_n$ which satisfy \eqref{eq: LMSVsimulation_normal} for
$\{\varepsilon_j,j\geq 1\}$ that follow a (non-centered) Pareto distribution with  parameter $\alpha$.
For the computation of the empirical power, we consider LMSV time series with a change-point of  height $h$ after a proportion $\tau$  of the data, i.e.
we  consider Pareto distributed  random variables $\varepsilon_j$, $j=1, \ldots, n$ with parameters $\alpha_j$, $j=1, \ldots, n$ such that $\alpha_j=\alpha$ for $j=1, \ldots, \lfloor n\tau\rfloor$ while  $\alpha_j=\alpha+h$ for $j=\lfloor n\tau\rfloor+1, \ldots, n$.
We computed the rejection rates of the testing procedures on the basis of $5,000$ realizations of time series with sample sizes $500$, $1,000$ and $2,000$.
The simulation results for $\tau=0.5$ are reported in Table \ref{table:cp_in_tail}; additional simulation results for $\tau=0.25$ can be found in Section S2 of the supplementary material.
The frequency of a type 1
error, i.e. the rejection rate under the hypothesis,
corresponds to the values in the
columns that are superscribed by \enquote{h = 0}.

\newgeometry{left=15mm,right=-15mm,top=15mm,bottom=35mm}
\begin{landscape}
\begin{table}[htbp]
\footnotesize 
 \begin{threeparttable}
\begin{tabular}{crcccccccccccccc}
 & & \multicolumn{7}{c}{\textbf{CUSUM}}  &  \multicolumn{7}{c}{\textbf{self-norm. CUSUM}} \\
\cline{4-9} \cline{11-16}\\
 & &
&  & $\alpha=0.5$ & &
& $\alpha=1$& & & &
$\alpha=0.5$& & &
$\alpha=1$\\
 $H$ & $n$ & & $h=0$ & $h=0.25$ & $h=0.5$ & $h=0$ & $h=0.25$ & $h=0.5$ & & $h=0$ & $h=0.25$ & $h=0.5$ & $h=0$ & $h=0.25$ & $h=0.5$ \\
\cline{1-16}\\
 0.6 & 500 &  &  0.457  &0.971 & 1.000 &  0.139& 0.312 & 0.597 &  &  0.035 & 0.603 & 0.907 &  0.040 & 0.137 & 0.318 \\
   & 1000 &  & 0.419 & 0.997 & 1.000 & 0.121 & 0.446 & 0.812 &  & 0.033 & 0.843 & 0.989 & 0.045 & 0.227 & 0.517 \\
   & 2000 &  & 0.388 & 1.000 & 1.000 &  0.113 & 0.628 & 0.953 &  &  0.038 & 0.974 & 1.000 &  0.045 & 0.363 & 0.732 \\
  0.7 & 500 &  & 0.213 & 0.794 & 0.976 &  0.084 & 0.161 & 0.296 &  &  0.026 & 0.364 & 0.692 & 0.041 & 0.085 & 0.167 \\
   & 1000 &  &  0.177& 0.916 & 0.998 &  0.071 & 0.188 & 0.402 &  &  0.038 & 0.551 & 0.869 & 0.043 & 0.120 & 0.248 \\
   & 2000 &  &  0.141 & 0.982 & 1.000 & 0.071 & 0.241 & 0.531 &  &  0.040 & 0.750 & 0.969 &  0.047 & 0.157 & 0.344 \\
  0.8 & 500 &  & 0.131 & 0.509 & 0.812 &  0.064 & 0.107 & 0.164 &  &  0.028 & 0.195 & 0.414 & 0.041 & 0.057 & 0.091 \\
   & 1000 &  & 0.107 & 0.611 & 0.902 &  0.064 & 0.109 & 0.192 &  &  0.034 & 0.296 & 0.539 & 0.047 & 0.065 & 0.106 \\
   & 2000 &  & 0.087 & 0.709 & 0.957 & 0.054 & 0.123 & 0.227 &  &  0.040 & 0.383 & 0.671 & 0.046 & 0.082 & 0.135 \\
  0.9 & 500 &  &  0.087 & 0.339 & 0.611 &  0.056 & 0.070 & 0.106 &  &  0.030 & 0.135 & 0.288 &  0.047 & 0.059 & 0.083 \\
   & 1000 &  & 0.075 & 0.360 & 0.656 &  0.054 & 0.068 & 0.121 &  & 0.038 & 0.181 & 0.350 & 0.050 & 0.059 & 0.092 \\
   & 2000 &  &  0.061& 0.391 & 0.716 &  0.051 & 0.071 & 0.127 &  & 0.046 & 0.216 & 0.410 & 0.052 & 0.060 & 0.085
\end{tabular}
 \captionsetup{font=footnotesize }
\caption{Rejection rates of the CUSUM tests for  LMSV time series   of length $n$ with Hurst parameter $H$ and a change in the tail index $\alpha$ of height $h$ after a proportion $\tau=0.5$. The calculations are based on 5,000 simulation runs. }
\end{threeparttable}
\label{table:cp_in_tail}
\end{table}
\end{landscape}
\restoregeometry

The finite sample performance of the testing procedures is affected by a
change of the tail parameter in the following way: lighter tails, i.e. an increase in the value of the tail parameter, go along with a decrease in power, but  an empirical size which is closer to the level of significance.
An increase of dependence, i.e. an increase of the Hurst parameter, results in a decrease of rejection rates.

We make the following observations with respect to the behavior of the testing procedures:
\begin{itemize}
\item The non-self-normalized test suffers from size distortions in so far as the rejection rates under the hypothesis are  considerably higher than the nominal level (especially for small values of $\alpha$ and $H$). 
\item All in all, the self-normalized test tends to be slightly undersized. 
\item
The  non-self-normalized test  has considerably better power than the self-normalized test for every parameter combination that has been considered.
\end{itemize}
 Due to the heavy size distortions of the non-self-normalized CUSUM test, it  seems advisable to choose the self-normalized test over the non-self-normalized test in the case of heavy tails characterized by small values of $\alpha$. However, for relatively light tails, the discrepancy between the test performances under the alternative suggests to choose   the non-self-normalized test over the self-normalized CUSUM test.

In all three test situations that have been considered,  we observe that under the hypothesis the rejection rates of the self-normalized tests yield better approximations of the nominal level than the non-self-normalized tests. On the other hand, considering self-normalized tests instead of applying a deterministic normalization leads to a loss of power.
In  the context of change-point tests, the so-called \enquote{better size, less power} phenomenon  has also  been observed by \cite{shao:2011},  \cite{shao:zhang:2010} and \cite{betken:2016}. Moreover, it is consistent with observations made in \cite{lobato:2001} and \cite{sun:phillips:jin:2008}. For a comparison of the testing procedures,
 it is important to note that the finite sample results  are based on simulations which were executed under the assumption that the normalization  of the non-self-normalized tests and the multiplicative quantities that  appear in the limits of the corresponding test statistics are known.
In particular, normalization and limit of the non-self-normalized statistics usually depend on $H$, $m$,  the slowly-varying function $L_{\gamma}$ that characterizes the autocovariances of the Gaussian random variables $Y_j$, $j\geq 1$,  the distribution of $\varepsilon_j$, $j\geq 1$, (or at least the tail parameter $\alpha$), and the function $\sigma$.
For all practical purposes  these quantities are unknown, and for this reason   have to be estimated. In contrast, the self-normalized test statistic can be computed from the given data   while its limit depends on the parameters $m$ and $H$ only. For an adequate comparison of the testing procedures this has to be taken into consideration.

Moreover, the simulation results show that the  choice of the change-point test depends on the particular test situation that is considered. In general, an application of Wilcoxon-type tests reduces the  influence of heavy tails in data generating processes on test decisions.
In fact, we have seen that Wilcoxon-based testing procedures yield better results than CUSUM tests when testing for changes in mean and variance.

\section*{Supplementary Material}

This section contains supplementary material. In Section \ref{sec:technical-proof} the   proof of Lemma 3.6 is presented; Section  \ref{sec:simulations}
comprises the full set of simulation results, i.e. in addition to the  results presented in the main document, the simulation tables in  
Section  \ref{sec:simulations}
contain the empirical power of the testing procedures for changes located after a proportion $\tau=0.25$ of the data.
\par

\setcounter{section}{0}
\setcounter{equation}{0}
\def\theequation{S\arabic{section}.\arabic{equation}}
\def\thesection{S\arabic{section}}

\section{Proof of Lemma 3.6}\label{sec:technical-proof}

Assume that  $\{X_j,j\geq 1\}$ follows the LMSV model.
 Define
\begin{align*}
M_n(x, t):=\sum\limits_{j=1}^{\lfloor nt\rfloor}\left(1_{\left\{\psi(X_j)\leq x\right\}}-\E\left(1_{\left\{\psi(X_j)\leq x\right\}}\left|\right. \mathcal{F}_{j-1}\right)\right).
\end{align*}
Lemma 3.6 states that
\begin{align*}
\frac{1}{\sqrt{n}}M_n(x, t)= \mathcal{O}_P(1)
\end{align*}
in $\spaceD(\left[-\infty,\infty\right]\times [0,1])$

\begin{proof}[Proof of Lemma 3.6]
In order to prove Lemma 3.6 we have to verify tightness in $\spaceD(\left[-\infty,\infty\right]\times [0,1])$.
For this we quote Theorem 3.1 in \cite{ivanoff1983}.
\begin{theorem}\label{thm:Ivanoff}
Let $\left(X_n, \mathcal{F}_n, P_n\right)$ be a sequence of probability spaces such that $X_n$ is a random element of  $\spaceD\left([0, 1]\times \left[0, 1\right]\right)$ for each $n$, and the process $X_n(x, t)$ is adapted to a complete, right continuous filtration $\mathcal{F}_n(x, t)\subseteq \mathcal{F}_n$.
 If $X_n(x, t)$ is tight  for each $(x, t)\in [0, 1]\times \left[0, 1\right]$ and if for all $\delta_n\searrow 0$ and each  canonical 1-stopping time $S_n$ and each canonical 2-stopping time $T_n$
\begin{align}
\sup\limits_{0\leq t\leq 1}\left|X_n(S_n +\delta_n, t)-X_n(S_n, t)\right|\overset{P}{\longrightarrow}0, \ \text{as} \ n\rightarrow \infty, \label{eq:Ivanoff_1}\\
\sup\limits_{0\leq x\leq 1}\left|X_n(x, T_n +\delta_n)-X_n(x, T_n)\right|\overset{P}{\longrightarrow}0, \ \text{as} \ n\rightarrow \infty, \label{eq:Ivanoff_2}
\end{align}
then $X_n$ is tight in  $\spaceD\left(\left[0, 1\right]\times \left[0, 1\right]\right)$.
\end{theorem}

In the following we will show that the conditions of  the theorem hold for  $X_n(x, t):=n^{-1/2}M_n\circ H$, where $H:\left[0, 1\right]\times
\left[0, 1\right]\longrightarrow \left[-\infty, \infty\right]\times
\left[0, 1\right]$ with $H(x,t)=(h(x), t)$ for some increasing isomorphism $h:\left[0, 1\right]\longrightarrow \left[-\infty, \infty\right]$.
Then,  Lemma 3.6 immediately follows from an application of Theorem \ref{thm:Ivanoff}. Indeed,
due to Lemma 3.5,  $n^{-1/2}M_n\circ H(x, t)$ is tight for each  $(x, t)\in \left[0, 1\right]\times \left[0, 1\right]$.

Recall that $\
\mathcal{F}_{j}=\sigma\left(\varepsilon_{j}, \varepsilon_{j-1}, \ldots, \eta_{j}, \eta_{j-1}, \ldots\right)$.
Define 
\begin{align*}
\mathcal{F}_n(x, t):=\mathcal{F}_{\lfloor nt\rfloor}
\end{align*}
for all $x\in [0, 1]$. Then,
 $X_n(x, t)$ is adapted to $\mathcal{F}_n(x, t)$.
Moreover, 
the corresponding   filtration is right-continuous.

Let $T_{n}$ denote a canonical $2$-stopping time  for $X_n(x, t)$. In particular, we require $\left\{T_n\leq t\right\}$ to be measurable with respect to $\mathcal{F}_n(1, t)=\mathcal{F}_{\lfloor nt\rfloor}$. Define $\tau_n:=\lfloor n T_n\rfloor$.

Note that
\begin{align*}
&\left|\frac{1}{\sqrt{n}}M_n(y, T_n +\delta_n)-\frac{1}{\sqrt{n}}M_n(y, T_n)\right|
=\left|\frac{1}{\sqrt{n}}\sum\limits_{j=\tau_n+1}^{\tau_n+\lfloor n\delta_n\rfloor}\zeta_j(y)\right|
\end{align*}
with 
\begin{align*}
\zeta_j(y)=1_{\left\{\psi(X_j)\leq y\right\}}-\E\left(1_{\left\{\psi(X_j)\leq y\right\}}\left|\right. \mathcal{F}_{j-1}\right).
\end{align*}
For $0\leq x\leq 1$  define 
\begin{align*}
\tilde{M}_n(x):=\frac{1}{\sqrt{n}}\sum\limits_{j=\tau_n+1}^{\tau_n+\lfloor n\delta_n\rfloor}\zeta_{j}(h(x)).
\end{align*}

In order  to show (\ref{eq:Ivanoff_2}), we have to prove that
\begin{align}\label{eq:toshow}
\sup\limits_{0\leq x\leq 1}\left|\tilde{M}_n(x)\right|\overset{P}{\longrightarrow}0, \ \text{as} \ n\rightarrow \infty.
\end{align}

Prior to the proof,  we establish the following result:
\begin{lemma}\label{lem: E_Var_2}
We have   
$\E\left[\tilde{M}_n(x)\right]=0$ for all $x\in \mathbb{R}$. For $x<y$
\begin{align*}
\Var\left(\tilde{M}_n(y)-\tilde{M}_n(x)\right)
&=\E\left[\left(\tilde{M}_n(y)-\tilde{M}_n(x)\right)^2
\right]\\
&\leq\frac{1}{n}\sum\limits_{j=1}^{\lfloor n\delta_n\rfloor}\left[
F_{\tilde{\psi}(X_{j+\tau_n})}(y)-F_{\tilde{\psi}(X_{j+\tau_n})}(x)
\right],
\end{align*}
where $F_{\tilde{\psi}(X_{j+\tau_n})}(x):=P\left(\tilde{\psi}(X_{j+\tau_n})\leq x\right)$ and $\tilde{\psi}:=h^{-1}\circ\psi$.
\end{lemma}

\begin{proof}
Define
\begin{align}
\alpha_j(y, x):=1_{\left\{h(x)<\psi(X_{j})\leq h(y)\right\}}-\E\left(1_{\left\{h(x)<\psi(X_{j})\leq h(y)\right\}}\left|\right. \mathcal{F}_{j-1}\right). \label{eq:alpha}
\end{align}
Then
\begin{align*}
\tilde{M}_n(y)-\tilde{M}_n(x)=\frac{1}{\sqrt{n}}\sum_{j=\tau_n+1}^{\tau_n+\lfloor n \delta_n\rfloor}\alpha_{j}(y,x).
\end{align*}
Since $\left\{\tau_n=k\right\}$ is measurable with respect to $\mathcal{F}_j$ for all $j\geq k$, it follows that 
\begin{align*}
&\E\left[\sum_{j=\tau_n+1}^{\tau_n+\lfloor n \delta_n\rfloor}\alpha_{j}(y, x)
\right]=\sum\limits_{k=1}^{n}\sum\limits_{j=1}^{\lfloor n\delta_n\rfloor}\E\left(\alpha_{j+k}(y, x)1_{\left\{\tau_n=k\right\}}\right)\\
&=\sum\limits_{k=1}^{n}\sum\limits_{j=1}^{\lfloor n\delta_n\rfloor}\E\left(
1_{\left\{x<\tilde{\psi}(X_{j+k})\leq y\right\}}1_{\left\{\tau_n=k\right\}}-\E\left(1_{\left\{x<\tilde{\psi}(X_{j+k})\leq y\right\}}1_{\left\{\tau_n=k\right\}}\left|\right. \mathcal{F}_{j+k-1}\right)
\right)\\
&=0, 
\end{align*}
where $\tilde{\psi}:=h^{-1}\circ \psi$.
Furthermore,
\begin{align*}
\E\left[\left(\sum_{j=\tau_n+1}^{\tau_n+\lfloor n\delta_n\rfloor}\alpha_{j}(y, x)\right)^2
\right]=\sum\limits_{k=1}^{n}\sum\limits_{i=1}^{\lfloor n\delta_n\rfloor}\sum\limits_{j=1}^{\lfloor n\delta_n\rfloor}\E\left(\alpha_{i+k}(y, x)\alpha_{j+k}(y, x)1_{\left\{\tau_n=k\right\}}\right).
\end{align*}
For $i<j$
\begin{align*}
&\alpha_{i+k}(y, x)\alpha_{j+k}(y, x)1_{\left\{\tau_n=k\right\}}\\
&=1_{\left\{x<\tilde{\psi}(X_{i+k})\leq  y\right\}}1_{\left\{x<\tilde{\psi}(X_{j+k})\leq  y\right\}}1_{\left\{\tau_n=k\right\}}\\
&\quad+\E\left(1_{\left\{x<\tilde{\psi}(X_{j+k})\leq  y\right\}}1_{\left\{\tau_n=k\right\}}\E\left(1_{\left\{x<\tilde{\psi}(X_{i+k})\leq  y\right\}}\left|\right. \mathcal{F}_{{i+k}-1}\right)\left|\right. \mathcal{F}_{j+k-1}\right)\\
&\quad-\E\left(1_{\left\{x<\tilde{\psi}(X_{i+k})\leq  y\right\}}1_{\left\{x<\tilde{\psi}(X_{j+k})\leq  y\right\}}1_{\left\{\tau_n=k\right\}}\left|\right. \mathcal{F}_{j+k-1}\right)\\
&\quad -1_{\left\{x<\tilde{\psi}(X_{j+k})\leq  y\right\}}1_{\left\{\tau_n=k\right\}}\E\left(1_{\left\{x<\tilde{\psi}(X_{i+k})\leq  y\right\}}\left|\right. \mathcal{F}_{i+k-1}\right).
\end{align*}
As a result,
\begin{align*}
\E\left(\alpha_{i+k}(y, x)\alpha_{j+k}(y, x)1_{\left\{\tau_n=k\right\}}\right)=0.
\end{align*}
Moreover, we have

\begin{align*}
&\sum\limits_{k=1}^{n}\E\left(\alpha^2_{i+k}( y, x)1_{\left\{\tau_n=k\right\}}\right)\\
&=\sum\limits_{k=1}^{n}\E\left[\left(1_{\left\{x<\tilde{\psi}(X_{i+k})\leq y\right\}}-\E\left(1_{\left\{x<\tilde{\psi}(X_{i+k})\leq y\right\}}\left|\right. \mathcal{F}_{i+k-1}\right)\right)^21_{\left\{\tau_n=k\right\}}\right]\\
&=\sum\limits_{k=1}^{n}\E\biggl[\left(1_{\left\{x<\tilde{\psi}(X_{i+k})\leq y\right\}}1_{\left\{\tau_n=k\right\}}-2 1_{\left\{x<\tilde{\psi}(X_{i+k})\leq y\right\}}\E\left(1_{\left\{x<\tilde{\psi}(X_{i+k})\leq y\right\}}1_{\left\{\tau_n=k\right\}}\left|\right. \mathcal{F}_{i+k-1}\right)\right)\\
&\quad  +\E\left(1_{\left\{x<\tilde{\psi}(X_{i+k})\leq y\right\}}\E\left(1_{\left\{x<\tilde{\psi}(X_{i+k})\leq y\right\}}1_{\left\{\tau_n=k\right\}}\left|\right. \mathcal{F}_{i+k-1}\right)\left|\right. \mathcal{F}_{i+k-1}\right)\biggr]\\
&=\sum\limits_{k=1}^{n}\left\{\E\left(1_{\left\{x<\tilde{\psi}(X_{i+k})\leq y\right\}}1_{\left\{\tau_n=k\right\}}\right)- \E\left(1_{\left\{x<\tilde{\psi}(X_{i+k})\leq y\right\}}\E\left(1_{\left\{x<\tilde{\psi}(X_{i+k})\leq y\right\}}1_{\left\{\tau_n=k\right\}}\left|\right. \mathcal{F}_{i+k-1}\right)\right)\right\}\\
&\leq\sum\limits_{k=1}^{n}\E\left[1_{\left\{x<\tilde{\psi}(X_{i+k})\leq y\right\}}1_{\left\{\tau_n=k\right\}}\right]
=\sum\limits_{k=1}^{n}\E\left[1_{\left\{x<\tilde{\psi}(X_{i+\tau_n})\leq y\right\}}1_{\left\{\tau_n=k\right\}}\right]\notag\\
&=\E\left[1_{\left\{x<\tilde{\psi}(X_{i+\tau_n})\leq y\right\}}\right]=F_{\tilde{\psi}(X_{i+\tau_n})}(y)-F_{\tilde{\psi}(X_{i+\tau_n})}(x).
\end{align*}
It follows that
\begin{align*}
&\Var\left(\sum_{j=\tau_n+1}^{\tau_n+\lfloor n \delta_n\rfloor}\alpha_j(y, x)\right)
=\E\left[\left(\sum_{j=\tau_n+1}^{\tau_n+\lfloor n \delta_n\rfloor}\alpha_{j}(y, x)\right)^2
\right]\leq\sum\limits_{j=1}^{\lfloor n\delta_n\rfloor}\left[
F_{\tilde{\psi}(X_{j+\tau_n})}(y)-F_{\tilde{\psi}(X_{j+\tau_n})}(x)
\right].
\end{align*}
\end{proof}

Due to Lemma \ref{lem: E_Var_2}
$\E (\tilde{M}_n(x))=0$
and
\begin{align*}
\Var (\tilde{M}_n(x))\leq \frac{1}{n}\sum\limits_{i=1}^{\lfloor n\delta_n\rfloor} F_{\tilde{\psi}(X_{i+\tau_n})}(x) \leq \delta_n\longrightarrow 0.
\end{align*}
Hence,
$\tilde{M}_n(x)$  converges to $0$ in probability. This implies fidi-convergence of $\tilde{M}_n$ as a process with values in $\spaceD\left[0, 1\right]$.

\paragraph{Proving \eqref{eq:Ivanoff_2}.}
In order to establish \eqref{eq:Ivanoff_2}, it remains to show tightness of $\tilde{M}_n$.
For this, we adopt the argument that  proves Theorem 15.6 in \cite{billingsley:1968}.
For any function $v$ in $\spaceD[0, 1]$ define the modulus $\omega^{''}_{v}(\delta)$ by

\begin{align*}
\omega^{''}_{v}(\delta)=\sup\min\left\{|v(x)-v(x_1)|, |v(x_2)-v(x)|\right\},
\end{align*}
where the supremum extends over $x_1$, $x$,  and $x_2$ satisfying
\begin{align*}
x_1\leq x\leq x_2, \ x_2-x_1\leq \delta.
\end{align*}
Under the assumption of fidi-convergence  it suffices to show that for each positive $\varepsilon$ and $\eta$, there exists a $\delta$, $0<\delta <1$, and an integer $n_0$ such that
 \begin{align*}
 P\left(\omega^{''}_{\tilde{M}_n}(\delta)\geq \varepsilon\right)\leq \eta, \ n\geq n_0,
 \end{align*}
(see Theorem 15.4 in \cite{billingsley:1968}).

Define
\begin{align*}
M_m^{''}:=\max\limits_{0\leq i\leq j\leq k\leq m}\min\left\{|S_j-S_i|, |S_k-S_j|\right\},
\end{align*}
where $S_i=\tilde{M}_n(\tau+\frac{i}{m}\delta)$.
Following the proof of  Theorem 12.5 in \cite{billingsley:1968}, there exists an $n_0\in \mathbb{N}$ such that
\begin{align}\label{inequ:M}
P\left(M_m^{''}\geq \lambda\right)\leq \frac{K}{\lambda^{2}}(u_1+\ldots +u_m)
\end{align}
holds for all positive $\lambda$, some constant $K$ and $n\geq n_0$,  if
\begin{align*}
P(\left\{|S_j-S_i|\geq \lambda, |S_k-S_j|\geq \lambda\right\})\leq \frac{\varepsilon_n}{\lambda^{2}}\sum_{i<l\leq k}u_l, \ 0\leq i\leq j\leq k\leq m,
\end{align*}
for some sequence $\varepsilon_n$, $n\in \mathbb{N}$, converging to $0$.
For $x_1\leq x\leq x_2$  it follows
by the Cauchy - Schwarz inequality for expected values and Lemma \ref{lem: E_Var_2} that

\begin{align}
&\E\left[\left|\tilde{M}_n(x_2)-\tilde{M}_n(x)\right|\left|\tilde{M}_n(x)-\tilde{M}_n(x_1)\right|\right]  \notag\\
&=\E\left[\frac{1}{\sqrt{n}}\left|\sum\limits_{j=\tau_n+1}^{\tau_n+\lfloor n\delta_n\rfloor}\alpha_j(x_2, x)\right|
\frac{1}{\sqrt{n}}\left|\sum\limits_{j=\tau_n+1}^{\tau_n+\lfloor n\delta_n\rfloor}\alpha_j(x, x_1)\right|\right]  \notag\\
&\leq \sqrt{\frac{1}{n}\E\left[\left(\sum_{j=\tau_n+1}^{\tau_n+\lfloor n\delta_n\rfloor}\alpha_j(x_2, x)\right)^2
\right]\frac{1}{n}\E\left[\left(\sum_{j=\tau_n+1}^{\tau_n+\lfloor n\delta_n \rfloor}\alpha_j(x, x_1)\right)^2\right]}\notag \\
&\leq \frac{1}{n}\sqrt{\sum\limits_{i=1}^{\lfloor n\delta_n\rfloor} \left(F_{\tilde{\psi}(X_{i+\tau_n})}(x)-F_{\tilde{\psi}(X_{i+\tau_n})}(x_1)\right)
}\sqrt{\sum\limits_{i=1}^{\lfloor n\delta_n\rfloor} \left(F_{\tilde{\psi}(X_{i+\tau_n})}(x_2)-F_{\tilde{\psi}(X_{i+\tau_n})}(x)\right)}\notag\\
&\leq \frac{1}{n}\sum\limits_{i=1}^{\lfloor n\delta_n\rfloor} \left(F_{\tilde{\psi}(X_{i+\tau_n})}(x_2)-F_{\tilde{\psi}(X_{i+\tau_n})}(x_1)\right)
.\label{eq:CS_inequality}
\end{align}
The Markov inequality yields
\begin{align*}
P\left(\left|S_j-S_i\right|\geq \lambda, \left|S_k-S_j\right|\geq \lambda\right)
&\leq P\left(\left|S_j-S_i\right|\left|S_k-S_j\right|\geq \lambda^2\right)\\
&\leq \frac{1}{\lambda^2}\E \left|S_j-S_i\right|\left|S_k-S_j\right|.
\end{align*}
Therefore, it follows  by \eqref{eq:CS_inequality} that
for some $\gamma\in (0, 1)$
\begin{align*}
&P\left(\left|S_j-S_i\right|\geq \lambda, \left|S_k-S_j\right|\geq \lambda\right)
\leq\frac{\delta_n^{\gamma}}{\lambda^2}\sum\limits_{i<l\leq k}u_l,
\end{align*}
where 
\begin{align*}
u_l:=\frac{\delta_n^{-\gamma}}{n}\sum\limits_{j=1}^{\lfloor n\delta_n\rfloor} \left(F_{\tilde{\psi}                 (X_{j+\tau_n})}\left(\tau+\frac{l}{m}\delta\right)-F_{\tilde{\psi}(X_{j+\tau_n})}\left(\tau+\frac{l-1}{m}\delta\right)\right).
\end{align*}
As a result, we  get
\begin{align}\label{inequ:M}
P\left(M_m^{''}
\geq \varepsilon\right)
\leq \frac{K}{\varepsilon^2}\frac{\delta_n^{-\gamma}}{n}\sum\limits_{j=1}^{\lfloor n\delta_n\rfloor} \left(F_{\tilde{\psi}(X_{j+\tau_n})}(\tau+\delta)-F_{\tilde{\psi}(X_{j+\tau_n})}(\tau)\right).
\end{align}
Define
\begin{align*}
\omega^{''}(\tilde{M}_n, [\tau, \tau +\delta]):=\sup\min\left\{|\tilde{M}_n(x)- \tilde{M}_n(x_1)|, |\tilde{M}_n(x_2)-\tilde{M}_n(x)|\right\},
\end{align*}
where the supremum extends over $x_1$, $x$, $x_2$ satisfying $\tau\leq x_1\leq x\leq x_2\leq \tau +\delta$.
Letting $m\longrightarrow \infty$ in \eqref{inequ:M}
yields
\begin{align}
P\left(\omega^{''}(\tilde{M}_n, [\tau, \tau +\delta])\geq \varepsilon\right)\leq \frac{K}{\varepsilon^2}\frac{\delta_n^{-\gamma}}{n}\sum\limits_{j=1}^{\lfloor n\delta_n\rfloor} \left(F_{\tilde{\psi}(X_{j+\tau_n})}(\tau+\delta)-F_{\tilde{\psi}(X_{j+\tau_n})}(\tau)\right)\label{eq:inequality_omega_tilde}
\end{align}
due to  right-continuity of  $\tilde{M}_n$.
Suppose that
$\delta=\frac{1}{u}$ for some integer $u$ and assume that
\begin{align}
&\omega^{''}(\tilde{M}_n, [2i\delta,  (2i+2)\delta])\leq \varepsilon,  \ 0\leq i\leq u-1,  \label{eq:interval1_tilde}
\\
&\omega^{''}(\tilde{M}_n, [(2i+1)\delta, (2i+3)\delta])\leq \varepsilon,  \ 0\leq i\leq u-2. \label{eq:interval2_tilde}
\end{align}
If $x_1\leq x \leq x_2$ and $x_2-x_1\leq \delta$, then $x_1$ and $x_2$ both lie in  one of the $2u-1$ intervals $[2i\delta, (2i+2)\delta]$, $0\leq i\leq u-1$,  $[(2i+1)\delta, (2i+3)\delta]$, $0\leq i\leq u-2$, so that
\begin{align*}
\min\left\{\left|\tilde{M}_n(x)-\tilde{M}_n(x_1)\right|, \left|\tilde{M}_n(x_2)-\tilde{M}_n(x)\right|\right\}\leq \varepsilon.
\end{align*}

Thus,  \eqref{eq:interval1_tilde} and \eqref{eq:interval2_tilde} together imply $\omega^{''}_{\tilde{M}_n}(\delta)\leq \varepsilon$.
It now follows by \eqref{eq:inequality_omega_tilde} that
\begin{align*}
P\left(\omega''_{\tilde{M}_n}(\delta)\geq\varepsilon\right)\leq\frac{K}{\varepsilon^2}\left(\Sigma' + \Sigma''\right),
\end{align*}
where each of $\Sigma'$ and $\Sigma ''$ is a sum of the form
\begin{align*}
&\sum\limits_{k=1}^r 
\frac{\delta_n^{-\gamma}}{n}\sum\limits_{j=1}^{\lfloor n\delta_n\rfloor} \left(F_{\tilde{\psi}(X_{j+\tau_n})}(x_k)-F_{\tilde{\psi}(X_{j+\tau_n})}(x_{k-1})\right)\\
&=
\frac{\delta_n^{-\gamma}}{n}\sum\limits_{j=1}^{\lfloor n\delta_n\rfloor}  \left(F_{\tilde{\psi}(X_{j+\tau_n})}(x_r)-F_{\tilde{\psi}(X_{j+\tau_n})}(x_{0})\right)\\
&\leq \delta_n^{1-\gamma}
\end{align*}
with $0\leq x_1\leq \ldots\leq x_r\leq 1$ and $x_k-x_{k-1}\leq 2\delta$.
Hence, we may conclude that
\begin{align*}
P\left(\omega''_{\tilde{M}_n}(\delta)\geq \varepsilon\right)\leq\frac{2K}{\varepsilon^2}\delta_n^{1-\gamma}.
\end{align*}
Since the right-hand side of the above inequality converges to $0$, \eqref{eq:Ivanoff_2}
has been proved.

\paragraph{Proving \eqref{eq:Ivanoff_1}.}
It remains to show \eqref{eq:Ivanoff_1}.
For this, let $S_{n}$ denote a canonical $1$-stopping time  for $X_n(x, t)$.
Note that
\begin{align*}
&\sup\limits_{0\leq t\leq 1}\left|X_n(S_n +\delta_n, t)-X_n(S_n, t)\right|=\sup\limits_{0\leq t\leq 1}\left|M^*_n(t)\right|
\end{align*}
with
\begin{align*}
M^*_n(t)&:=\frac{1}{\sqrt{n}}\sum\limits_{j=1}^{\lfloor nt\rfloor}\left(1_{\left\{h(S_n )<\psi(X_j)\leq h(S_n +\delta_n)\right\}}-\E\left(1_{\left\{h(S_n)<\psi(X_j)\leq h(S_n +\delta_n)\right\}}\left|\right. \mathcal{F}_{j-1}\right)\right)\\
&=\frac{1}{\sqrt{n}}\sum\limits_{j=1}^{\lfloor nt\rfloor}\left(1_{\left\{0< \tilde{\psi}(X_j)-S_n\leq \delta_n\right\}}-\E\left(1_{\left\{0<\tilde{\psi}(X_j)-S_n\leq \delta_n\right\}}\left|\right. \mathcal{F}_{j-1}\right)\right), 
\end{align*}
where $\tilde{\psi}:=h^{-1}\circ \psi$.

Prior to the proof of \eqref{eq:Ivanoff_1}, we establish the following result:
\begin{lemma}\label{lem: E_Var}
We have  $\E( M_n^*(t))=0$ for all $t\in \left[0, 1\right]$. For $s<t$
\begin{align*}
\Var (M_n^*(t)-M_n^*(s))&=\E\left[\left(M_n^*(t)-M_n^*(s)\right)^2\right]\\
&\leq \frac{1}{n}\sum\limits_{i=\lfloor ns\rfloor +1}^{\lfloor nt\rfloor}\left(F_{\tilde{\psi}(X_i)-S_n}(\delta_n)-F_{\tilde{\psi}(X_i)-S_n}(0)\right),
\end{align*}
where $F_{\tilde{\psi}(X_i)-S_n}(x):=P\left(\tilde{\psi}(X_i)-S_n\leq x\right)$.
\end{lemma}

\begin{proof}
Note that
\begin{align*}
&M_n^*(t)-M_n^*(s)
=\frac{1}{\sqrt{n}}\sum\limits_{j=\lfloor ns\rfloor + 1}^{\lfloor nt\rfloor}\alpha_j(\delta_n, 0),
\end{align*}
where
\begin{align}
\alpha_j(y, x):=1_{\left\{x<\tilde{\psi}(X_j)-S_n\leq y\right\}}-\E\left(1_{\left\{x<\tilde{\psi}(X_j)-S_n\leq y\right\}}\left|\right. \mathcal{F}_{j-1}\right). \label{eq:alpha}
\end{align}
Obviously,
\begin{align*}
\E\left(\sum_{j=\lfloor ns\rfloor +1}^{\lfloor n t\rfloor}\alpha_j(\delta_n, 0)\right)=0.
\end{align*}
Furthermore,
\begin{align*}
\Var\left(\sum_{j=\lfloor ns\rfloor +1}^{\lfloor n t\rfloor}\alpha_j(\delta_n, 0)\right)
&=\E\left[\left(\sum_{j=\lfloor ns\rfloor +1}^{\lfloor n t\rfloor}\alpha_j(\delta_n, 0)\right)^2
\right] \\
&=\sum\limits_{i=\lfloor ns\rfloor +1}^{\lfloor nt\rfloor}\sum\limits_{j=\lfloor ns\rfloor +1}^{\lfloor nt\rfloor}\E\left(\alpha_i(\delta_n, 0)\alpha_j(\delta_n, 0)\right).
\end{align*}
 The canonical $1$-stopping time $S_n$ takes on only countably many values $s_i$, $i\in \mathbb{N}$; see \cite{ivanoff1983}.
Note that for $i<j$
\begin{align*}
\alpha_i(\delta_n, 0)\alpha_j(\delta_n, 0)
&=\sum_{k=1}^{\infty}\alpha_i(\delta_n, 0)\alpha_j(\delta_n, 0)1_{\left\{S_n=s_k\right\}}\\
&=\sum_{k=1}^{\infty}1_{\left\{S_n=s_k\right\}}\biggl\{ 1_{\left\{0<\tilde{\psi}(X_i)-s_k\leq  \delta_n\right\}}1_{\left\{0<\tilde{\psi}(X_j)-s_k\leq  \delta_n\right\}}\\
&\quad +\E\left(1_{\left\{0<\tilde{\psi}(X_j)-s_k\leq  \delta_n\right\}}\E\left(1_{\left\{0<\tilde{\psi}(X_i)-s_k\leq  \delta_n\right\}}\left|\right. \mathcal{F}_{i-1}\right)\left|\right. \mathcal{F}_{j-1}\right)\\
&\quad-\E\left(1_{\left\{0<\tilde{\psi}(X_i)-s_k\leq  \delta_n\right\}}1_{\left\{0<\tilde{\psi}(X_j)-s_k\leq  \delta_n\right\}}\left|\right. \mathcal{F}_{j-1}\right)\\
&\quad -1_{\left\{0<\tilde{\psi}(X_j)-s_k\leq  \delta_n\right\}}\E\left(1_{\left\{0<\tilde{\psi}(X_i)-s_k\leq  \delta_n\right\}}\left|\right. \mathcal{F}_{i-1}\right)\biggr\}\\
&=1_{\left\{0<\tilde{\psi}(X_i)-S_n\leq  \delta_n\right\}}1_{\left\{0<\tilde{\psi}(X_j)-S_n\leq  \delta_n\right\}}\\
&\quad+\E\left(1_{\left\{0<\tilde{\psi}(X_j)-S_n\leq  \delta_n\right\}}\E\left(1_{\left\{0<\tilde{\psi}(X_i)-S_n\leq  \delta_n\right\}}\left|\right. \mathcal{F}_{i-1}\right)\left|\right. \mathcal{F}_{j-1}\right)\\
&\quad-\E\left(1_{\left\{0<\tilde{\psi}(X_i)-S_n\leq  \delta_n\right\}}1_{\left\{0<\tilde{\psi}(X_j)-S_n\leq  \delta_n\right\}}\left|\right. \mathcal{F}_{j-1}\right)\\
&\quad -1_{\left\{0<\tilde{\psi}(X_j)-S_n\leq  \delta_n\right\}}\E\left(1_{\left\{0<\tilde{\psi}(X_i)-S_n\leq  \delta_n\right\}}\left|\right. \mathcal{F}_{i-1}\right).
\end{align*}
As a result, $\E\left(\alpha_i(\delta_n, 0)\alpha_j(\delta_n, 0)\right)=0$.
Moreover, 
\begin{align*}
&\E\left(\alpha^2_j(\delta_n, 0)\right)\\
&=\E\left(1_{\left\{0<\tilde{\psi}(X_j)-S_n\leq  \delta_n\right\}}\right)-2\E\left(1_{\left\{0<\tilde{\psi}(X_j)-S_n\leq  \delta_n\right\}}\E\left(1_{\left\{0<\tilde{\psi}(X_j)-S_n\leq  \delta_n\right\}}\left|\right. \mathcal{F}_{j-1}\right)\right)\\
&\quad+\E\left(\E\left(1_{\left\{0<\tilde{\psi}(X_j)-S_n\leq  \delta_n\right\}}\E\left(1_{\left\{0<\tilde{\psi}(X_j)-S_n\leq  \delta_n\right\}}\left|\right. \mathcal{F}_{j-1}\right)\left|\right. \mathcal{F}_{j-1}\right)\right)\\
&=\E\left(1_{\left\{0<\tilde{\psi}(X_j)-S_n\leq  \delta_n\right\}}\right)-\E\left(1_{\left\{0<\tilde{\psi}(X_j)-S_n\leq  \delta_n\right\}}\E\left(1_{\left\{0<\tilde{\psi}(X_j)-S_n\leq  \delta_n\right\}}\left|\right. \mathcal{F}_{j-1}\right)\right)\\
&\leq 
\E\left(1_{\left\{0<\tilde{\psi}(X_j)-S_n\leq  \delta_n\right\}}\right)
=F_{\tilde{\psi}(X_j)-S_n}(\delta_n)-F_{\tilde{\psi}(X_j)-S_n}(0).
\end{align*}
As a result,
\begin{align*}
\Var\left(\sum_{j=\lfloor ns\rfloor + 1}^{\lfloor n t\rfloor}\alpha_j(\delta_n, 0)\right)
&=\sum\limits_{j=\lfloor ns\rfloor + 1}^{\lfloor nt\rfloor}\E\left(\alpha_j^2(\delta_n, 0)\right)\\
&\leq \sum\limits_{j=\lfloor ns\rfloor +1}^{\lfloor nt\rfloor}\left(F_{\tilde{\psi}(X_j)-S_n}(\delta_n)-F_{\tilde{\psi}(X_j)-S_n}(0)\right).
\end{align*}
\end{proof}

Lemma \ref{lem: E_Var} yields
$\E (M^*_n(t))=0$
and
\begin{align}\label{eq:Var}
\Var M^*_n(t)&\leq  t \E\left(\sup\limits_{0\leq x\leq 1}\left(F_{\lfloor nt\rfloor}(\delta_n+x)-F_{\lfloor nt\rfloor}(x)\right)\right), 
\end{align}
where $F_{\lfloor nt\rfloor}$ denotes the empirical distribution function of $\tilde{\psi}(X_i)$, $i\geq 1$, i.e.
\begin{align*}
F_{\lfloor nt\rfloor}(x):=\frac{1}{\lfloor nt\rfloor}\sum\limits_{i=1}^{\lfloor nt\rfloor}
1_{\left\{\tilde{\psi}(X_i)\leq x\right\}}.
\end{align*}
Note that $\sup_{0\leq x\leq 1}\left(F_{\lfloor nt\rfloor}(\delta_n+x)-F_{\lfloor nt\rfloor}(x)\right) $ converges to $0$ almost surely due to the Glivenko-Cantelli theorem for stationary, ergodic sequences and since $F_{\tilde{\psi}(X_1)}$ is a continuous distribution function. It follows that its expected value converges to $0$, as well. 
Therefore, the right-hand side of   \eqref{eq:Var}
converges to $0$.
Thus,
 $M^*_n(t)$  converges to $0$ in probability. This implies fidi-convergence  of $M^*_n$ as a process with values in $\spaceD\left[0, 1\right]$.
 Again, we make use of Theorem 15.4 in \cite{billingsley:1968}, i.e. we verify
 that for each positive $\varepsilon$ and $\eta$, there exists a $\delta$, $0<\delta <1$, and an integer $n_0$ such that
 \begin{align*}
 P\left(\omega^{''}_{M^*_n}(\delta)\geq \varepsilon\right)\leq \eta, \ n\geq n_0.
 \end{align*}

Define
\begin{align*}
M_m^{''}:=\max\limits_{0\leq i\leq j\leq k\leq m}\min\left\{|S_j-S_i|, |S_k-S_j|\right\},
\end{align*}
where $S_i=M^*_n(\tau+\frac{i}{m}\delta)$.

For $t_1\leq t\leq t_2$ we have
\begin{align*}
&\left|M^*_n(t)-M^*_n(t_1)\right|=\frac{1}{\sqrt{n}}\left|\sum\limits_{j=\lfloor nt_1\rfloor+1}^{\lfloor n t\rfloor}\alpha_j(\delta_n, 0)\right|,\\
&\left|M^*_n(t_2)-M^*_n(t)\right|
=\frac{1}{\sqrt{n}}\left|\sum\limits_{j=\lfloor nt\rfloor +1}^{\lfloor n t_2 \rfloor}\alpha_j(\delta_n, 0)\right|.
\end{align*}
The Cauchy-Schwarz inequality yields
\begin{align*}
&\E\left|M^*_n(t)-M^*_n(t_1)\right|\left|M^*_n(t_2)-M^*_n(t)\right|\\
&\leq\sqrt{\E\left[\frac{1}{n}\left(\sum\limits_{j=\lfloor n t_1\rfloor +1}^{\lfloor nt\rfloor}\alpha_j( \delta_n, 0)\right)^2\right]\E\left[
\frac{1}{n}\left(\sum\limits_{j=\lfloor nt\rfloor +1}^{\lfloor n t_2\rfloor}\alpha_j(\delta_n, 0)\right)^2\right]}\\
&\leq \frac{1}{n}\sum\limits_{i=\lfloor nt_1\rfloor +1}^{\lfloor nt_2\rfloor}\left(F_{\tilde{\psi}(X_i)-S_n}(\delta_n)-F_{\tilde{\psi}(X_i)-S_n}(0)\right).
\end{align*}
By the same argument as in the proof of \eqref{eq:Ivanoff_2} it follows that
\begin{align} \label{inequ:M_2}
P\left(M_m^{''}
\geq \varepsilon\right)
\leq 
\frac{K}{\varepsilon^2}\frac{\gamma_n^{-1}}{n}\sum\limits_{j=\lfloor n\tau\rfloor +1}^{\lfloor n(\tau+\delta)\rfloor} \left(F_{\tilde{\psi}(X_{j})-S_n}(\delta_n)-F_{\tilde{\psi}(X_{j})-S_n}(0)\right)
\end{align}
for any sequence $\gamma_n$, $n\in \mathbb{N}$, converging to $0$. 
Taking    the right-continuity of $M_n^*$ into consideration, we may, as before,  conclude that
\begin{align*}
P\left(\omega''_{M^*_n}(\delta)\geq\varepsilon\right)&\leq\frac{2K}{\varepsilon^2}\gamma_n^{-1}\frac{1}{n}\sum\limits_{j=1}^{n} \left(F_{\tilde{\psi}(X_{j})-S_n}(\delta_n)-F_{\tilde{\psi}(X_{j})-S_n}(0)\right)\\
&\leq \frac{2K}{\varepsilon^2}\gamma_n^{-1}\E\left(\sup\limits_{0\leq x\leq 1}\left(F_{n}(\delta_n+x)-F_{n}(x)\right)\right).
\end{align*}
If we choose $\gamma_n$, $n\in \mathbb{N}$, such that 
\begin{align*}
\gamma_n^{-1}\E\left(\sup\limits_{0\leq x\leq 1}\left(F_{n}(\delta_n+x)-F_{n}(x)\right)\right)
\end{align*}
converges to $0$, 
the right-hand side of the above inequality vanishes as $n$ tends to $\infty$ due to the choice of $\gamma
_n$. This concludes the proof of \eqref{eq:Ivanoff_1} as well as the proof of Lemma 3.6.
\end{proof}

\section{Simulations}\label{sec:simulations}

We computed the empirical size and the empirical power of the testing procedures on the basis of simulated time series 
according to the corresponding specifications in
Section 6 of the main document.
In addition to the simulation results referred to in the main document, the simulation tables contain the empirical power of the testing procedures when the change point is located closer to the testing region ($\tau=0.25$).
The rejection rates are reported in Tables \ref{table:cp_in_mean_normal_eps}
to 
\ref{table:cp_in_tail}.

%\subsection{Change in the mean}

%\subsection{Change in the variance}

%\subsection{Change in the tail}

\newgeometry{left=15mm,right=-15mm,top=15mm,bottom=35mm}
\begin{landscape}
\begin{table}[htbp]
\footnotesize
 \begin{threeparttable}
\begin{tabular}{ccrcccccccccccc}
& & & & \multicolumn{5}{c}{\textbf{CUSUM}}  &  &  \multicolumn{5}{c}{\textbf{self-norm. CUSUM}} \\
 \cline{5-9} \cline{11-15}\\
 &{$H$} & $n$   &  & $h=0$ &  &{$h = 0.5$} &{$h = 1$} &{$h = 2$} &  & $h=0$ &  &{$h = 0.5$} &{$h = 1$} &{$h = 2$} \\
 \multirow{11}{*}{ \rotatebox{90}{$\tau=0.25$}} & 0.6 & 500 &  & {0.046} &  & 0.214 & 0.803 & 1.000 &  & {0.043} &  & 0.201 & 0.597 & 0.951 \\
 &  & 1000 &  & {0.045} &  & 0.452 & 0.987 & 1.000 &  & {0.047} &  & 0.352 & 0.818 & 0.996 \\
 &  & 2000 &  & {0.043} &  & 0.816 & 1.000 & 1.000 &  & {0.044} &  & 0.575 & 0.957 & 1.000 \\
 & 0.7 & 500 &  & {0.064} &  & 0.222 & 0.792 & 1.000 &  & {0.044} &  & 0.223 & 0.594 & 0.951 \\
 &  & 1000 &  & {0.056} &  & 0.461 & 0.989 & 1.000 &  & {0.044} &  & 0.361 & 0.823 & 0.993 \\
 &  & 2000 &  & {0.050} &  & 0.815 & 1.000 & 1.000 &  & {0.046} &  & 0.583 & 0.960 & 1.000 \\
 & 0.8 & 500 &  & {0.074} &  & 0.204 & 0.822 & 1.000 &  & {0.039} &  & 0.264 & 0.643 & 0.944 \\
 &  & 1000 &  & {0.074} &  & 0.435 & 0.987 & 1.000 &  & {0.044} &  & 0.403 & 0.827 & 0.989 \\
 &  & 2000 &  & {0.072} &  & 0.814 & 1.000 & 1.000 &  & {0.048} &  & 0.621 & 0.948 & 0.999 \\
 & 0.9 & 500 &  & {0.095} &  & 0.190 & 0.849 & 1.000 &  & {0.045} &  & 0.391 & 0.720 & 0.939 \\
 &  & 1000 &  & {0.101} &  & 0.370 & 0.991 & 1.000 &  & {0.045} &  & 0.550 & 0.858 & 0.976 \\
 &  & 2000 &  & {0.101} &  & 0.847 & 1.000 & 1.000 &  & {0.043} &  & 0.694 & 0.931 & 0.997 \\
 &  & {} &  &  &  & {} & {} & {} &  &  &  & {} & {} & {} \\
 \multirow{11}{*}{ \rotatebox{90}{$\tau=0.5$}} & 0.6 & 500 &  &  &  & 0.384 & 0.958 & 1.000 &  &  &  & 0.372 & 0.853 & 0.998 \\
 &  & 1000 &  &  &  & 0.728 & 1.000 & 1.000 &  &  &  & 0.627 & 0.975 & 1.000 \\
 &  & 2000 &  &  &  & 0.958 & 1.000 & 1.000 &  &  &  & 0.856 & 0.998 & 1.000 \\
 &  0.7 & 500 &  &  &  & 0.392 & 0.962 & 1.000 &  &  &  & 0.395 & 0.858 & 0.995 \\
 &  & 1000 &  &  &  & 0.736 & 0.999 & 1.000 &  &  &  & 0.629 & 0.974 & 1.000 \\
 &  & 2000 &  &  &  & 0.964 & 1.000 & 1.000 &  &  &  & 0.869 & 0.997 & 1.000 \\
 & 0.8 & 500 &  &  &  & 0.369 & 0.969 & 1.000 &  &  &  & 0.462 & 0.860 & 0.993 \\
 &  & 1000 &  &  &  & 0.734 & 1.000 & 1.000 &  &  &  & 0.670 & 0.963 & 0.999 \\
 &  & 2000 &  &  &  & 0.966 & 1.000 & 1.000 &  &  &  & 0.851 & 0.996 & 1.000 \\
 & 0.9 & 500 &  &  &  & 0.311 & 0.977 & 1.000 &  &  &  & 0.583 & 0.875 & 0.980 \\
 &  & 1000 &  &  &  & 0.739 & 0.998 & 1.000 &  &  &  & 0.736 & 0.946 & 0.996 \\
 &  & 2000 &  &  &  & 0.979 & 1.000 & 1.000 &  &  &  & 0.860 & 0.987 & 1.000
 \end{tabular}
  \captionsetup{font=footnotesize }
\caption{Rejection rates of the CUSUM tests for LMSV  time series  (standard normal $\varepsilon_j,$ $j\geq 1$) of length $n$
 with Hurst parameter $H$  and a shift in the mean of height $h$ after a proportion $\tau$.
The calculations are based on 5,000 simulation runs.}
\label{table:cp_in_mean_normal_eps}
\end{threeparttable}
\end{table}
\end{landscape}
\restoregeometry

\newgeometry{left=15mm,right=-15mm,top=15mm,bottom=35mm}
\begin{landscape}
\begin{table}[htbp]
\footnotesize
 \begin{threeparttable}
\begin{tabular}{ccrcccccccccccccc}
& & & \multicolumn{7}{c}{\textbf{CUSUM}}  &  \multicolumn{7}{c}{\textbf{Wilcoxon}} \\
 \cline{5-10} \cline{12-17}\\
& & &
& &$\alpha=2.5$ & &
& $\alpha=4$& & & &
$\alpha=2.5$& & &
$\alpha=4$\\
& $H$ & $n$ & & $h=0$ & $h=0.5$ & $h=1$ & $h=0$ & $h=0.5$ & $h=1$ & & $h=0$ & $h=0.5$ & $h=1$ & $h=0$ & $h=0.5$ & $h=1$ \\
\cline{2-17}\\
 \multirow{11}{*}{ \rotatebox{90}{$\tau=0.25$}}&
0.6 & 500 & & 0.035 & 0.081 & 0.625 & 0.047 & 0.986 & 1.000 & &0.628 & 1.000 & 1.000 & 0.802 & 1.000 & 1.000 \\
& & 1000 & &{0.034} & 0.239 & 0.969 & {0.051} & 1.000 & 1.000 & &{0.578} & 1.000 & 1.000 &{0.751} & 1.000 & 1.000 \\
& & 2000 & &{0.034} & 0.638 & 1.000 &{0.048} & 1.000 & 1.000 &  &{0.524} & 1.000 & 1.000 &{0.713} & 1.000 & 1.000 \\
& $0.7$ &  500 & &{0.035} & 0.092 & 0.621 &{0.053} & 0.986 & 1.000 & &{0.331} & 1.000 & 1.000 &{0.475} & 1.000 & 1.000 \\
& & 1000 & &{0.037} & 0.239 & 0.961 &{0.058} & 1.000 & 1.000 & &{0.270} & 1.000 & 1.000 &{0.384} & 1.000 & 1.000 \\
& & 2000 & &{0.039} & 0.648 & 0.999 &{0.051} & 1.000 & 1.000 & &{0.207} & 1.000 & 1.000 &{0.300} & 1.000 & 1.000 \\
& 0.8 & 500 & &{0.045} & 0.096 & 0.622 &{0.073} & 0.984 & 1.000 & &{0.191} & 0.987 & 1.000 &{0.273} & 1.000 & 1.000 \\
& & 1000 &&{0.041} & 0.224 & 0.966 &{0.066} & 1.000 & 1.000 & &{0.144} & 0.997 & 1.000 &{0.187} & 1.000 & 1.000 \\
& & 2000 & &{0.042} & 0.637 & 0.999 &{0.066} & 1.000 & 1.000 & &{0.108} & 1.000 & 1.000 &{0.132} & 1.000 & 1.000 \\
&0.9 & 500 & &{0.057} & 0.100 & 0.627 &{0.080} & 0.986 & 1.000 & &{0.188} & 0.919 & 0.993 &{0.232} & 0.994 & 1.000 \\
& & 1000 & &{0.059} & 0.207 & 0.973 &{0.090} & 1.000 & 1.000 & &{0.139} & 0.929 & 0.995 &{0.165} & 0.999 & 1.000 \\
& & 2000 & &{0.064} & 0.635 & 1.000 &{0.092} & 1.000 & 1.000 & &{0.108} & 0.954 & 0.997 &{0.121} & 0.999 & 1.000 \\
\\
 \multirow{11}{*}{ \rotatebox{90}{$\tau=0.5$}}
 &0.6 & 500 & & & 0.048 & 0.213 &  & 0.866 & 1.000 &  & & 1.000 & 1.000 &  & 1.000 & 1.000 \\
& & 1000 & &  & 0.085 & 0.745 &  & 0.993 & 1.000 & & & 1.000 & 1.000 &  & 1.000 & 1.000 \\
& & 2000 & & & 0.288 & 0.986 &  & 1.000 & 1.000 &  & & 1.000 & 1.000 &  & 1.000 & 1.000 \\
&0.7 & 500 & &  & 0.050 & 0.209 &  & 0.864 & 1.000 & & & 1.000 & 1.000 &  & 1.000 & 1.000 \\
& & 1000 &  & & 0.087 & 0.752 &  & 0.994 & 1.000 & & & 1.000 & 1.000 &  & 1.000 & 1.000 \\
& & 2000 & &  & 0.285 & 0.983 &  & 1.000 & 1.000 &  & & 1.000 & 1.000 &  & 1.000 & 1.000 \\
&0.8 & 500 & &  & 0.055 & 0.207 &  & 0.879 & 1.000 & & & 0.974 & 1.000 &  & 1.000 & 1.000 \\
& & 1000 &  & & 0.108 & 0.757 &  & 0.994 & 1.000 &  & & 0.994 & 1.000 &  & 1.000 & 1.000 \\
& & 2000 &  & & 0.280 & 0.983 &  & 1.000 & 1.000 &  & & 1.000 & 1.000 &  & 1.000 & 1.000 \\
&0.9 & 500 &  & & 0.069 & 0.191 &  & 0.901 & 1.000 &  & & 0.863 & 0.984 &  & 0.995 & 1.000 \\
& & 1000 &  & & 0.111 & 0.783 &  & 0.992 & 1.000 &  & & 0.888 & 0.990 &  & 0.996 & 1.000 \\
& & 2000 &  & & 0.238 & 0.984 &  & 1.000 & 1.000 &  & & 0.917 & 0.996 &   &0.998 & 1.000
\end{tabular}
 \captionsetup{font=footnotesize }
\caption{Rejection rates of the  CUSUM and Wilcoxon  test for  LMSV  time series (Pareto distributed $\varepsilon_j,$ $j\geq 1$)   of length $n$ with Hurst parameter $H$, tail index $\alpha$  and a shift in the mean of height $h$ after a proportion $\tau$. The calculations are based on 5,000 simulation runs.}
\label{table:cp_in_mean}
 \end{threeparttable}
\end{table}
\end{landscape}
\restoregeometry

\newgeometry{left=15mm,right=-15mm,top=15mm,bottom=35mm}
\begin{landscape}
\begin{table}[htbp]
\footnotesize
 \begin{threeparttable}
\begin{tabular}{ccrcccccccccccccc}
& & & \multicolumn{7}{c}{\textbf{self-norm. CUSUM}}  &  \multicolumn{7}{c}{\textbf{self-norm. Wilcoxon}} \\
 \cline{5-10} \cline{12-17}\\
& & &
& &$\alpha=2.5$ & &
& $\alpha=4$& & & &
$\alpha=2.5$& & &
$\alpha=4$\\
& $H$ & $n$ & & $h=0$ & $h=0.5$ & $h=1$ & $h=0$ & $h=0.5$ & $h=1$ & & $h=0$ & $h=0.5$ & $h=1$ & $h=0$ & $h=0.5$ & $h=1$ \\
\cline{2-17}\\
 \multirow{11}{*}{ \rotatebox{90}{$\tau=0.25$}} & 0.6 & 500 &  & {0.046} & 0.181 & 0.539 & {0.042} & 0.688 & 0.958 &  & {0.032} & 0.879 & 0.991 & {0.030} & 0.995 & 1.000 \\
 &  & 1000 &  & {0.049} & 0.290 & 0.722 & {0.044} & 0.862 & 0.990 &  & {0.032} & 0.973 & 1.000 & {0.030} & 1.000 & 1.000 \\
 &  & 2000 &  & {0.053} & 0.458 & 0.875 & {0.026} & 0.967 & 0.999 &  & {0.034} & 0.999 & 1.000 & {0.028} & 1.000 & 1.000 \\
 & 0.7 & 500 &  & {0.051} & 0.204 & 0.552 & {0.042} & 0.697 & 0.954 &  & {0.029} & 0.680 & 0.938 & {0.021} & 0.960 & 0.997 \\
 &  & 1000 &  & {0.050} & 0.295 & 0.727 & {0.046} & 0.866 & 0.990 &  & {0.032} & 0.856 & 0.988 & {0.027} & 0.993 & 1.000 \\
 &  & 2000 &  & {0.049} & 0.455 & 0.868 & {0.042} & 0.966 & 0.998 &  & {0.037} & 0.948 & 0.999 & {0.030} & 0.999 & 1.000 \\
 & 0.8 & 500 &  & {0.045} & 0.226 & 0.580 & {0.044} & 0.720 & 0.951 &  & {0.031} & 0.424 & 0.772 & {0.021} & 0.815 & 0.964 \\
 &  & 1000 &  & {0.042} & 0.338 & 0.736 & {0.040} & 0.870 & 0.989 &  & {0.033} & 0.559 & 0.862 & {0.024} & 0.915 & 0.984 \\
 &  & 2000 &  & {0.050} & 0.498 & 0.881 & {0.052} & 0.960 & 0.998 &  & {0.034} & 0.673 & 0.938 & {0.023} & 0.958 & 0.998 \\
 & 0.9 & 500 &  & {0.044} & 0.329 & 0.645 & {0.041} & 0.760 & 0.947 &  & {0.031} & 0.309 & 0.582 & {0.020} & 0.640 & 0.861 \\
 &  & 1000 &  & {0.051} & 0.446 & 0.761 & {0.042} & 0.871 & 0.980 &  & {0.039} & 0.369 & 0.650 & {0.034} & 0.734 & 0.912 \\
 &  & 2000 &  & {0.041} & 0.585 & 0.869 & {0.048} & 0.949 & 0.996 &  & {0.049} & 0.422 & 0.719 & {0.039} & 0.791 & 0.947 \\
 &  & {} &  &  & {} & {} &  & {} & {} &  &  & {} & {} &  & {} & {} \\
 \multirow{11}{*}{ \rotatebox{90}{$\tau=0.5$}} & 0.6 & 500 &  &  & 0.384 & 0.801 &  & 0.904 & 0.990 &  &  & 0.994 & 1.000 &  & 1.000 & 1.000 \\
 &  & 1000 &  &  & 0.564 & 0.909 &  & 0.973 & 0.998 &  &  & 1.000 & 1.000 &  & 1.000 & 1.000 \\
 &  & 2000 &  &  & 0.744 & 0.962 &  & 0.993 & 1.000 &  &  & 1.000 & 1.000 &  & 1.000 & 1.000 \\
 & 0.7 & 500 &  &  & 0.401 & 0.801 &  & 0.902 & 0.989 &  &  & 0.950 & 1.000 &  & 1.000 & 1.000 \\
 &  & 1000 &  &  & 0.565 & 0.904 &  & 0.972 & 0.998 &  &  & 0.993 & 1.000 &  & 1.000 & 1.000 \\
 &  & 2000 &  &  & 0.744 & 0.966 &  & 0.994 & 0.999 &  &  & 1.000 & 1.000 &  & 1.000 & 1.000 \\
 & 0.8 & 500 &  &  & 0.424 & 0.804 &  & 0.899 & 0.990 &  &  & 0.776 & 0.977 &  & 0.987 & 0.999 \\
 &  & 1000 &  &  & 0.589 & 0.905 &  & 0.966 & 0.997 &  &  & 0.896 & 0.995 &  & 0.998 & 1.000 \\
 &  & 2000 &  &  & 0.761 & 0.959 &  & 0.994 & 0.999 &  &  & 0.963 & 0.999 &  & 1.000 & 1.000 \\
 & 0.9 & 500 &  &  & 0.527 & 0.815 &  & 0.893 & 0.982 &  &  & 0.622 & 0.890 &  & 0.912 & 0.990 \\
 &  & 1000 &  &  & 0.650 & 0.898 &  & 0.959 & 0.997 &  &  & 0.708 & 0.936 &  & 0.956 & 0.996 \\
 &  & 2000 &  &  & 0.781 & 0.954 &  & 0.989 & 0.999 &  &  & 0.779 & 0.960 &  & 0.976 & 0.998 \\
\end{tabular}
 \captionsetup{font=footnotesize }
\caption{Rejection rates of the self-normalized  CUSUM and the self-normalized Wilcoxon  test for  LMSV time series  (Pareto distributed $\varepsilon_j,$ $j\geq 1$) of length $n$ with Hurst parameter $H$, tail index $\alpha$  and a shift in the mean of height $h$ after a proportion $\tau$. The calculations are based on 5,000 simulation runs.}
\label{table:cp_in_mean_(SN)}
\end{threeparttable}
\end{table}
\end{landscape}
\restoregeometry

\newgeometry{left=15mm,right=-15mm,top=15mm,bottom=35mm}
\begin{landscape}
\begin{table}[htbp]
\footnotesize
 \begin{threeparttable}
\begin{tabular}{ccrcccccccccccccc}
& & & \multicolumn{7}{c}{\textbf{CUSUM}}  &  \multicolumn{7}{c}{\textbf{Wilcoxon}} \\
 \cline{5-10} \cline{12-17}\\
& & &
& &$\alpha=4.5$ & &
& $\alpha=6$& & & &
$\alpha=4.5$& & &
$\alpha=6$\\
& $H$ & $n$ & & $h=1$ & $h=0.5$ & $h=2$ & $h=1$ & $h=0.5$ & $h=2$ & & $h=1$ & $h=0.5$ & $h=2$ & $h=1$ & $h=0.5$ & $h=2$ \\
\cline{2-17}\\
 \multirow{11}{*}{ \rotatebox{90}{$\tau=0.25$}} & 0.6 & 500 &  & \multicolumn{1}{r}{0.464} & 0.252 & 0.963 & \multicolumn{1}{r}{0.467} & 0.270 & 0.978 &  & \multicolumn{1}{r}{0.119} & 0.925 & 0.937 & \multicolumn{1}{r}{0.130} & 0.929 & 0.929 \\
 &  & 1000 &  & \multicolumn{1}{r}{0.589} & 0.383 & 0.995 & \multicolumn{1}{r}{0.596} & 0.385 & 0.997 &  & \multicolumn{1}{r}{0.112} & 0.994 & 0.995 & \multicolumn{1}{r}{0.118} & 0.995 & 0.996 \\
 &  & 2000 &  & \multicolumn{1}{r}{0.708} & 0.529 & 1.000 & \multicolumn{1}{r}{0.694} & 0.574 & 1.000 &  & \multicolumn{1}{r}{0.108} & 1.000 & 1.000 & \multicolumn{1}{r}{0.104} & 1.000 & 1.000 \\
 & 0.7 & 500 &  & \multicolumn{1}{r}{0.330} & 0.164 & 0.852 & \multicolumn{1}{r}{0.330} & 0.174 & 0.882 &  & \multicolumn{1}{r}{0.078} & 0.584 & 0.587 & \multicolumn{1}{r}{0.078} & 0.606 & 0.591 \\
 &  & 1000 &  & \multicolumn{1}{r}{0.374} & 0.197 & 0.937 & \multicolumn{1}{r}{0.404} & 0.207 & 0.961 &  & \multicolumn{1}{r}{0.066} & 0.781 & 0.784 & \multicolumn{1}{r}{0.068} & 0.781 & 0.780 \\
 &  & 2000 &  & \multicolumn{1}{r}{0.431} & 0.263 & 0.983 & \multicolumn{1}{r}{0.443} & 0.273 & 0.991 &  & \multicolumn{1}{r}{0.060} & 0.932 & 0.929 & \multicolumn{1}{r}{0.061} & 0.934 & 0.936 \\
 & 0.8 & 500 &  & \multicolumn{1}{r}{0.235} & 0.116 & 0.670 & \multicolumn{1}{r}{0.244} & 0.111 & 0.686 &  & \multicolumn{1}{r}{0.074} & 0.314 & 0.328 & \multicolumn{1}{r}{0.074} & 0.332 & 0.319 \\
 &  & 1000 &  & \multicolumn{1}{r}{0.258} & 0.131 & 0.770 & \multicolumn{1}{r}{0.256} & 0.132 & 0.786 &  & \multicolumn{1}{r}{0.067} & 0.398 & 0.386 & \multicolumn{1}{r}{0.067} & 0.398 & 0.382 \\
 &  & 2000 &  & \multicolumn{1}{r}{0.275} & 0.139 & 0.837 & \multicolumn{1}{r}{0.271} & 0.137 & 0.861 &  & \multicolumn{1}{r}{0.060} & 0.502 & 0.499 & \multicolumn{1}{r}{0.059} & 0.505 & 0.501 \\
 & 0.9 & 500 &  & \multicolumn{1}{r}{0.179} & 0.088 & 0.470 & \multicolumn{1}{r}{0.170} & 0.084 & 0.486 &  & \multicolumn{1}{r}{0.088} & 0.251 & 0.254 & \multicolumn{1}{r}{0.088} & 0.252 & 0.262 \\
 &  & 1000 &  & \multicolumn{1}{r}{0.184} & 0.100 & 0.513 & \multicolumn{1}{r}{0.177} & 0.097 & 0.523 &  & \multicolumn{1}{r}{0.078} & 0.267 & 0.277 & \multicolumn{1}{r}{0.079} & 0.272 & 0.272 \\
 &  & 2000 &  & \multicolumn{1}{r}{0.191} & 0.101 & 0.564 & \multicolumn{1}{r}{0.191} & 0.099 & 0.566 &  & \multicolumn{1}{r}{0.072} & 0.283 & 0.283 & \multicolumn{1}{r}{0.065} & 0.279 & 0.277 \\
 &  & \multicolumn{1}{l}{} &  &  & \multicolumn{1}{l}{} & \multicolumn{1}{l}{} &  & \multicolumn{1}{l}{} & \multicolumn{1}{l}{} &  &  & \multicolumn{1}{l}{} & \multicolumn{1}{l}{} &  & \multicolumn{1}{l}{} & \multicolumn{1}{l}{} \\
 \multirow{11}{*}{ \rotatebox{90}{$\tau=0.5$}} & 0.6 & 500 &  &  & 0.377 & 0.954 &  & 0.417 & 0.967 &  &  & 0.988 & 0.990 &  & 0.990 & 0.990 \\
 &  & 1000 &  &  & 0.565 & 0.992 &  & 0.594 & 0.996 &  &  & 1.000 & 1.000 &  & 1.000 & 1.000 \\
 &  & 2000 &  &  & 0.774 & 0.999 &  & 0.814 & 0.999 &  &  & 1.000 & 1.000 &  & 1.000 & 1.000 \\
 & 0.7 & 500 &  &  & 0.252 & 0.821 &  & 0.252 & 0.839 &  &  & 0.808 & 0.808 &  & 0.806 & 0.809 \\
 &  & 1000 &  &  & 0.313 & 0.932 &  & 0.333 & 0.946 &  &  & 0.934 & 0.937 &  & 0.936 & 0.935 \\
 &  & 2000 &  &  & 0.416 & 0.984 &  & 0.436 & 0.987 &  &  & 0.992 & 0.992 &  & 0.990 & 0.991 \\
 & 0.8 & 500 &  &  & 0.170 & 0.623 &  & 0.166 & 0.651 &  &  & 0.529 & 0.525 &  & 0.515 & 0.528 \\
 &  & 1000 &  &  & 0.196 & 0.733 &  & 0.199 & 0.757 &  &  & 0.605 & 0.615 &  & 0.624 & 0.631 \\
 &  & 2000 &  &  & 0.226 & 0.838 &  & 0.230 & 0.848 &  &  & 0.741 & 0.738 &  & 0.746 & 0.734 \\
 & 0.9 & 500 &  &  & 0.127 & 0.440 &  & 0.132 & 0.445 &  &  & 0.410 & 0.419 &  & 0.404 & 0.406 \\
 &  & 1000 &  &  & 0.137 & 0.487 &  & 0.140 & 0.506 &  &  & 0.442 & 0.440 &  & 0.452 & 0.424 \\
 &  & 2000 &  &  & 0.135 & 0.542 &  & 0.148 & 0.540 &  &  & 0.466 & 0.467 &  & 0.468 & 0.473
\end{tabular}
 \captionsetup{font=footnotesize }
\caption{Rejection rates of the  CUSUM and Wilcoxon test for  LMSV  time series   of length $n$ with Hurst parameter $H$, tail index $\alpha$  and a shift in the variance of height $h^2$ after a proportion $\tau$. The calculations are based on 5,000 simulation runs.}
\label{table:cp_in_var}
\end{threeparttable}
\end{table}
\end{landscape}
\restoregeometry

\newgeometry{left=15mm,right=-15mm,top=15mm,bottom=35mm}
\begin{landscape}
\begin{table}[htbp]
\footnotesize
 \begin{threeparttable}
\begin{tabular}{ccrcccccccccccccc}
& & & \multicolumn{7}{c}{\textbf{self-norm. CUSUM}}  &  \multicolumn{7}{c}{\textbf{self-norm. Wilcoxon}} \\
 \cline{5-10} \cline{12-17}\\
& & &
& &$\alpha=4.5$ & &
& $\alpha=6$& & & &
$\alpha=4.5$& & &
$\alpha=6$\\
& $H$ & $n$ & & $h=1$ & $h=0.5$ & $h=2$ & $h=1$ & $h=0.5$ & $h=2$ & & $h=1$ & $h=0.5$ & $h=2$ & $h=1$ & $h=0.5$ & $h=2$ \\
\cline{2-17}\\
 \multirow{11}{*}{ \rotatebox{90}{$\tau=0.25$}} & 0.6 & 500 &  & \multicolumn{1}{r}{0.035} & 0.229 & 0.040 & \multicolumn{1}{r}{0.038} & 0.237 & 0.043 &  & \multicolumn{1}{r}{0.040} & 0.518 & 0.495 & \multicolumn{1}{r}{0.041} & 0.511 & 0.494 \\
 &  & 1000 &  & \multicolumn{1}{r}{0.033} & 0.291 & 0.043 & \multicolumn{1}{r}{0.035} & 0.317 & 0.050 &  & \multicolumn{1}{r}{0.043} & 0.736 & 0.739 & \multicolumn{1}{r}{0.047} & 0.731 & 0.745 \\
 &  & 2000 &  & \multicolumn{1}{r}{0.033} & 0.383 & 0.048 & \multicolumn{1}{r}{0.034} & 0.423 & 0.060 &  & \multicolumn{1}{r}{0.044} & 0.901 & 0.897 & \multicolumn{1}{r}{0.042} & 0.908 & 0.905 \\
 & 0.7 & 500 &  & \multicolumn{1}{r}{0.022} & 0.146 & 0.022 & \multicolumn{1}{r}{0.024} & 0.161 & 0.025 &  & \multicolumn{1}{r}{0.046} & 0.252 & 0.248 & \multicolumn{1}{r}{0.048} & 0.263 & 0.259 \\
 &  & 1000 &  & \multicolumn{1}{r}{0.022} & 0.192 & 0.023 & \multicolumn{1}{r}{0.025} & 0.225 & 0.027 &  & \multicolumn{1}{r}{0.044} & 0.380 & 0.381 & \multicolumn{1}{r}{0.046} & 0.380 & 0.378 \\
 &  & 2000 &  & \multicolumn{1}{r}{0.020} & 0.263 & 0.030 & \multicolumn{1}{r}{0.021} & 0.296 & 0.034 &  & \multicolumn{1}{r}{0.049} & 0.526 & 0.528 & \multicolumn{1}{r}{0.050} & 0.531 & 0.536 \\
 & 0.8 & 500 &  & \multicolumn{1}{r}{0.015} & 0.086 & 0.013 & \multicolumn{1}{r}{0.019} & 0.098 & 0.015 &  & \multicolumn{1}{r}{0.038} & 0.120 & 0.119 & \multicolumn{1}{r}{0.042} & 0.136 & 0.127 \\
 &  & 1000 &  & \multicolumn{1}{r}{0.015} & 0.107 & 0.017 & \multicolumn{1}{r}{0.016} & 0.128 & 0.017 &  & \multicolumn{1}{r}{0.049} & 0.159 & 0.155 & \multicolumn{1}{r}{0.045} & 0.165 & 0.156 \\
 &  & 2000 &  & \multicolumn{1}{r}{0.015} & 0.140 & 0.014 & \multicolumn{1}{r}{0.017} & 0.171 & 0.016 &  & \multicolumn{1}{r}{0.047} & 0.198 & 0.192 & \multicolumn{1}{r}{0.046} & 0.197 & 0.198 \\
 & 0.9 & 500 &  & \multicolumn{1}{r}{0.019} & 0.079 & 0.018 & \multicolumn{1}{r}{0.023} & 0.088 & 0.021 &  & \multicolumn{1}{r}{0.048} & 0.100 & 0.101 & \multicolumn{1}{r}{0.049} & 0.097 & 0.096 \\
 &  & 1000 &  & \multicolumn{1}{r}{0.019} & 0.099 & 0.015 & \multicolumn{1}{r}{0.021} & 0.112 & 0.020 &  & \multicolumn{1}{r}{0.053} & 0.109 & 0.107 & \multicolumn{1}{r}{0.050} & 0.110 & 0.105 \\
 &  & 2000 &  & \multicolumn{1}{r}{0.020} & 0.123 & 0.020 & \multicolumn{1}{r}{0.023} & 0.133 & 0.023 &  & \multicolumn{1}{r}{0.056} & 0.118 & 0.114 & \multicolumn{1}{r}{0.047} & 0.104 & 0.123 \\
 &  & \multicolumn{1}{l}{} &  &  & \multicolumn{1}{l}{} & \multicolumn{1}{l}{} &  & \multicolumn{1}{l}{} & \multicolumn{1}{l}{} &  &  & \multicolumn{1}{l}{} & \multicolumn{1}{l}{} &  & \multicolumn{1}{l}{} & \multicolumn{1}{l}{} \\
 \multirow{11}{*}{ \rotatebox{90}{$\tau=0.5$}} & 0.6 & 500 &  &  & 0.145 & 0.152 &  & 0.158 & 0.164 &  &  & 0.816 & 0.831 &  & 0.823 & 0.816 \\
 &  & 1000 &  &  & 0.206 & 0.193 &  & 0.227 & 0.221 &  &  & 0.965 & 0.963 &  & 0.962 & 0.960 \\
 &  & 2000 &  &  & 0.277 & 0.270 &  & 0.324 & 0.318 &  &  & 0.998 & 0.998 &  & 0.996 & 0.996 \\
 & 0.7 & 500 &  &  & 0.090 & 0.090 &  & 0.100 & 0.093 &  &  & 0.524 & 0.530 &  & 0.534 & 0.521 \\
 &  & 1000 &  &  & 0.125 & 0.121 &  & 0.130 & 0.135 &  &  & 0.701 & 0.695 &  & 0.698 & 0.699 \\
 &  & 2000 &  &  & 0.165 & 0.168 &  & 0.196 & 0.193 &  &  & 0.859 & 0.863 &  & 0.854 & 0.863 \\
 & 0.8 & 500 &  &  & 0.045 & 0.046 &  & 0.053 & 0.054 &  &  & 0.270 & 0.274 &  & 0.270 & 0.278 \\
 &  & 1000 &  &  & 0.063 & 0.064 &  & 0.062 & 0.070 &  &  & 0.357 & 0.361 &  & 0.350 & 0.361 \\
 &  & 2000 &  &  & 0.073 & 0.086 &  & 0.099 & 0.083 &  &  & 0.439 & 0.443 &  & 0.454 & 0.454 \\
 & 0.9 & 500 &  &  & 0.044 & 0.037 &  & 0.047 & 0.040 &  &  & 0.200 & 0.205 &  & 0.195 & 0.198 \\
 &  & 1000 &  &  & 0.053 & 0.052 &  & 0.064 & 0.058 &  &  & 0.228 & 0.224 &  & 0.246 & 0.228 \\
 &  & 2000 &  &  & 0.066 & 0.065 &  & 0.076 & 0.074 &  &  & 0.263 & 0.248 &  & 0.248 & 0.264
\end{tabular}
 \captionsetup{font=footnotesize }
\caption{Rejection rates of the self-normalized CUSUM and the self-normalized Wilcoxon  test for  LMSV  time series   of length $n$ with Hurst parameter $H$, tail index $\alpha$  and a shift in the variance of height $h^2$ after a proportion $\tau$. The calculations are based on 5,000 simulation runs.}
\label{table:cp_in_var (SN)}
\end{threeparttable}
\end{table}
\end{landscape}
\restoregeometry

\newgeometry{left=15mm,right=-15mm,top=15mm,bottom=35mm}
\begin{landscape}
\begin{table}[htbp]
\footnotesize
 \begin{threeparttable}
\begin{tabular}{ccrcccccccccccccc}
& & & \multicolumn{7}{c}{\textbf{CUSUM}}  &  \multicolumn{7}{c}{\textbf{self-norm. CUSUM}} \\
\cline{5-10} \cline{12-17}\\
& & &
&  & $\alpha=0.5$ & &
& $\alpha=1$& & & &
$\alpha=0.5$& & &
$\alpha=1$\\
& $H$ & $n$ & & $h=0$ & $h=0.25$ & $h=0.5$ & $h=0$ & $h=0.25$ & $h=0.5$ & & $h=0$ & $h=0.25$ & $h=0.5$ & $h=0$ & $h=0.25$ & $h=0.5$ \\
\cline{2-17}\\
 \multirow{11}{*}{ \rotatebox{90}{$\tau=0.25$}}  & 0.6 & 500 &  & {0.457} & 0.884 & 0.994 & {0.139} & 0.215 & 0.393 &  & {0.035} & 0.373 & 0.725 & {0.040} & 0.087 & 0.181 \\
 &  & 1000 &  & {0.419} & 0.982 & 1.000 & {0.121} & 0.298 & 0.588 &  & {0.033} & 0.596 & 0.912 & {0.045} & 0.124 & 0.277 \\
 &  & 2000 &  & {0.388} & 0.999 & 1.000 & {0.113} & 0.436 & 0.825 &  & {0.038} & 0.819 & 0.985 & {0.045} & 0.206 & 0.453 \\
 & 0.7 & 500 &  & {0.213} & 0.601 & 0.879 & {0.084} & 0.108 & 0.184 &  & {0.026} & 0.189 & 0.435 & {0.041} & 0.065 & 0.100 \\
 &  & 1000 &  & {0.177} & 0.761 & 0.977 & {0.071} & 0.134 & 0.249 &  & {0.038} & 0.294 & 0.596 & {0.043} & 0.079 & 0.137 \\
 &  & 2000 &  & {0.141} & 0.907 & 0.998 & {0.071} & 0.151 & 0.350 &  & {0.040} & 0.452 & 0.774 & {0.047} & 0.098 & 0.191 \\
 & 0.8 & 500 &  & {0.131} & 0.329 & 0.590 & {0.064} & 0.081 & 0.105 &  & {0.028} & 0.092 & 0.194 & {0.041} & 0.046 & 0.056 \\
 &  & 1000 &  & {0.107} & 0.379 & 0.716 & {0.064} & 0.081 & 0.115 &  & {0.034} & 0.130 & 0.276 & {0.047} & 0.053 & 0.064 \\
 &  & 2000 &  & {0.087} & 0.491 & 0.836 & {0.054} & 0.083 & 0.138 &  & {0.040} & 0.177 & 0.350 & {0.046} & 0.050 & 0.078 \\
 & 0.9 & 500 &  & {0.087} & 0.201 & 0.376 & {0.056} & 0.058 & 0.079 &  & {0.030} & 0.073 & 0.133 & {0.047} & 0.044 & 0.058 \\
 &  & 1000 &  & {0.075} & 0.216 & 0.425 & {0.054} & 0.068 & 0.080 &  & {0.038} & 0.083 & 0.157 & {0.050} & 0.055 & 0.060 \\
 &  & 2000 &  & {0.061} & 0.221 & 0.480 & {0.051} & 0.061 & 0.074 &  & {0.046} & 0.101 & 0.189 & {0.052} & 0.055 & 0.058 \\
\\
 \multirow{11}{*}{ \rotatebox{90}{$\tau=0.5$}}  & 0.6 & 500 &  &  & 0.971 & 1.000 &  & 0.312 & 0.597 &  &  & 0.603 & 0.907 &  & 0.137 & 0.318 \\
 &  & 1000 &  &  & 0.997 & 1.000 &  & 0.446 & 0.812 &  &  & 0.843 & 0.989 &  & 0.227 & 0.517 \\
 &  & 2000 &  &  & 1.000 & 1.000 &  & 0.628 & 0.953 &  &  & 0.974 & 1.000 &  & 0.363 & 0.732 \\
 & 0.7 & 500 &  &  & 0.794 & 0.976 &  & 0.161 & 0.296 &  &  & 0.364 & 0.692 &  & 0.085 & 0.167 \\
 &  & 1000 &  &  & 0.916 & 0.998 &  & 0.188 & 0.402 &  &  & 0.551 & 0.869 &  & 0.120 & 0.248 \\
 &  & 2000 &  &  & 0.982 & 1.000 &  & 0.241 & 0.531 &  &  & 0.750 & 0.969 &  & 0.157 & 0.344 \\
 & 0.8 & 500 &  &  & 0.509 & 0.812 &  & 0.107 & 0.164 &  &  & 0.195 & 0.414 &  & 0.057 & 0.091 \\
 &  & 1000 &  &  & 0.611 & 0.902 &  & 0.109 & 0.192 &  &  & 0.296 & 0.539 &  & 0.065 & 0.106 \\
 &  & 2000 &  &  & 0.709 & 0.957 &  & 0.123 & 0.227 &  &  & 0.383 & 0.671 &  & 0.082 & 0.135 \\
 & 0.9 & 500 &  &  & 0.339 & 0.611 &  & 0.070 & 0.106 &  &  & 0.135 & 0.288 &  & 0.059 & 0.083 \\
 &  & 1000 &  &  & 0.360 & 0.656 &  & 0.068 & 0.121 &  &  & 0.181 & 0.350 &  & 0.059 & 0.092 \\
 &  & 2000 &  &  & 0.391 & 0.716 &  & 0.071 & 0.127 &  &  & 0.216 & 0.410 &  & 0.060 & 0.085
\end{tabular}
 \captionsetup{font=footnotesize }
\caption{Rejection rates of the CUSUM tests for  LMSV time series   of length $n$ with Hurst parameter $H$ and a change in the tail index $\alpha$ of height $h$ after a proportion $\tau$. The calculations are based on 5,000 simulation runs. }
\label{table:cp_in_tail}
\end{threeparttable}
\end{table}
\end{landscape}
\restoregeometry

\singlespace
\small
\bibliographystyle{apalike}
\bibliography{bib-changepoint}

\end{document}